\documentclass[11pt,leqno]{amsart}
\usepackage{amssymb,amsfonts,amscd, color}
\usepackage{amsfonts,amsmath,amssymb,amscd,subfigure}
\usepackage{srcltx} 
\usepackage{latexsym} 
\usepackage{esint}
\usepackage{multicol}
\usepackage{mathrsfs}
\usepackage{algorithm}
\usepackage{bm}
\usepackage[plainpages=false,pdfpagelabels,hypertexnames=false,colorlinks=true,
pdfstartview=FitV,linkcolor=red,citecolor=blue,urlcolor=blue]{hyperref}

\def\nc{\newcommand}

\def\om{\omega}

 \def\Om{\Omega}

\nc\pa{\partial}

\nc\CC{\mathbb{C}}
\nc\RR{\mathbb{R}}
\nc\QQ{\mathbb{Q}}
\nc\ZZ{\mathbb{Z}}
\nc\NN{\mathbb{N}}

\nc\m[1]{\left| #1\right|}
\nc\norm[1]{\left\|#1\right\|}
\newtheorem{theorem}{Theorem}[section]
\newtheorem{lemma}[theorem]{Lemma}
\newtheorem{corollary}[theorem]{Corollary}
\newtheorem{proposition}[theorem]{Proposition}
\newtheorem{definition}[theorem]{Definition}
\newtheorem{remark}[theorem]{Remark}        
\numberwithin{equation}{section}

\begin{document}

\title[Muckenhoupt-Wheeden type  bounds and applications]
{Nonlinear Muckenhoupt-Wheeden type  bounds on Reifenberg flat domains, with applications to quasilinear Riccati type equations}

\author[Nguyen Cong Phuc]
{Nguyen Cong Phuc$^{*}$}
\address{Department of Mathematics,
Louisiana State University,
303 Lockett Hall, Baton Rouge, LA 70803, USA.}
\email{pcnguyen@math.lsu.edu}


\thanks{$^{*}$Supported in part by NSF grant DMS-0901083.}

\begin{abstract}
A weighted norm inequality of Muckenhoupt-Wheeden type is obtained for gradients of solutions to a class of quasilinear equations with measure data on Reifenberg flat domains.
This essentially leads to a resolution of an existence  problem for quasilinear 
Riccati type equations with a  gradient source term of arbitrary power law growth.
\end{abstract}

\maketitle

\section{Introduction}\label{Introduction}

In this paper we first obtain a weighted norm inequality of Muckenhoupt-Wheeden type  for gradients of solution to the Dirichlet problem with measure data
\begin{eqnarray}\label{basicpde}
\left\{ \begin{array}{rcl}
 -\text{div}\,\mathcal{A}(x, \nabla u)&=& \mu \quad \text{in} ~\Omega, \\
u&=&0  \quad \text{on}~ \partial \Omega.
\end{array}\right.
\end{eqnarray} 

Given a finite measure $\mu$ on a bounded domain $\Om\subset\RR^n$, $n\geq 2$, under certain natural conditions on $\mathcal{A}$ and $\Om$, we  show in particular  that solutions $u$ of 
\eqref{basicpde} satisfy the bound
\begin{equation}\label{MWbound}
\int_{\Om} |\nabla u|^q w dx \leq C \int_{\Om} {\rm\bf M}_1(\mu)^{\frac{q}{p-1}} w dx
\end{equation}
for any $q\in (0, \infty)$ and any weight $w$ in the $A_\infty$ class. Here $p\in (2-1/n, n]$ is  the natural growth of $\mathcal{A}$ (modeled after the $p$-Laplacian), and ${\rm\bf M}_1$ is the first order fractional maximal function defined for each locally finite  measure $\nu$ by
$${\rm\bf M}_1(\nu)(x)= \sup_{r>0} \frac{r\, |\nu|(B_r(x))}{|B_r(x)|},\qquad x\in \RR^n.$$

Inequality \eqref{MWbound} is then applied to obtain a sharp existence result for the Riccati type equation
\begin{eqnarray*}
\left\{ \begin{array}{rcl}
 -{\rm div}\,\mathcal{A}(x, \nabla u)&=& |\nabla u|^q + \mu \quad \text{in} ~\Om, \\
u&=&0  \quad \text{on}~ \partial \Omega,
\end{array}\right.
\end{eqnarray*} 
with a source term growth $q>p-1$. This in particular confirms a conjecture of Igor E. Verbitsky (personal communication), which has also been recently stated as an open problem  in \cite{VHV},
pages 13--14.

Specifically, the nonlinearity   $\mathcal{A}: \RR^n\times\RR^n \rightarrow \RR^n$ in \eqref{basicpde} is a Carath\'edory  function, i.e., $\mathcal{A}(x,\xi)$ is measurable in $x$ for every $\xi$ and continuous in 
$\xi$ for a.e. $x$. Moreover,  for a.e. $x$, $\mathcal{A}(x,\xi)$ is differentiable in $\xi$ away from the origin. Let  $\mathcal{A}_{\xi}(x,\xi)$ denote 
its  Jacobian matrix for $\xi\in\RR^n\setminus\{0\}$.
For our purpose, we also assume that $\mathcal{A}$ satisfies the following growth and ellipticity conditions: for some $2-1/n<p\leq n$ there holds

\begin{equation}\label{sublinear}
|\mathcal{A}(x,\xi)| \leq\beta\m{\xi}^{p-1},
\end{equation}
and 
\begin{equation}\label{ellipticity}
\langle \mathcal{A}_{\xi}(x,\xi)\lambda, \lambda\rangle \geq \alpha  |\xi|^{p-2} |\lambda|^{2}, \quad |\mathcal{A}_{\xi}(x,\xi)| \leq\beta\m{\xi}^{p-2}
\end{equation}
for every $(\lambda,\xi)\in\RR^n \times\RR^n\setminus\{(0,0)\}$ and a.e. $x \in \RR^n$. Here $\alpha$ and $\beta$ are positive structural constants.

Note that \eqref{sublinear} and the Carath\'edory property imply that $\mathcal{A}(x,0)=0$ for a.e. $x\in \RR^n$, whereas
the first condition in \eqref{ellipticity} implies the following strict monotonicity condition:
\begin{equation}\label{monotone}
  \langle\mathcal{A}(x,\xi)-\mathcal{A}(x,\eta),\xi-\eta \rangle\geq
 c(p,\alpha)(|\xi|^2+|\eta|^2)^{\frac{p-2}{2}}|\xi-\eta|^2 
\end{equation}
for any  $(\xi, \eta)\in \RR^n \times \RR^n\setminus (0,0)$ and for almost every $x \in\RR^n$ (see, e.g., \cite{Tol}). 

For the purpose of this paper we also require that the nonlinearity $\mathcal{A}$ satisfy a smallness condition of BMO type in the $x$-variable.
We call such a condition the $(\delta, R_0)$-BMO condition.

\begin{definition} \label{bmodefinition} 
 We say that $\mathcal{A}({x, \xi})$ satisfies a $(\delta, R_0)$-BMO condition for some $\delta, R_0>0$ with exponent $s>0$, if
\begin{equation*}
[\mathcal{A}]^{R_0}_{ s}:=\sup_{y\in\RR^n, \, 0<r\leq R_0 }  \left(\fint_{B_{r}(y)}\Upsilon(\mathcal{A},
B_{r}(y))(x)^{s}dx\right)^{\frac{1}{s}} \leq \delta,
\end{equation*}
where 
\[
\Upsilon(\mathcal{A}, B_r(y)) (x) := \sup_{\xi \in \mathbb{R}^{n}\setminus \{0\}} \frac{|\mathcal{A}({x, \xi}) - 
\overline{\mathcal{A}}_{B_r(y)}({\xi})|}{|\xi|^{p-1}}
\]
with $\overline{\mathcal{A}}_{B_r(y)}(\xi)$ denoting the average of $\mathcal{A}(\cdot, \xi)$ over the ball $B_r(y)$, i.e.,
\begin{equation*}
\overline{\mathcal{A}}_{B_r(y)}(\xi) := \fint_{B_r(y)}\mathcal{A}(x, \xi)dx=\frac{1}{|B_r(y)|}\int_{B_r(y)}\mathcal{A}(x, \xi)dx.
\end{equation*}
\end{definition}

A typical example of such a nonlinearity $\mathcal{A}$ is given by $\mathcal{A}(x, \xi)=|\xi|^{p-2}\xi$ which gives rise to the standard $p$-Laplacian $\Delta_p u={\rm div}(|\nabla u|^{p-2}\nabla u)$. 

In the case  $p=2$, the above $(\delta, R_0)$-BMO condition 
was introduced in \cite{BW2}, whereas  such a condition for  general $p\in(1,\infty)$  appears  in the recent paper \cite{Ph3}.
We remark that the  $(\delta, R_0)$-BMO condition allows $\mathcal{A}(x, \xi)$ has discontinuity in $x$ and it can be used as an 
appropriate substitute for the Sarason \cite{Sa} VMO  condition   (see, e.g., \cite{BW1, BW2,  IKM, Mil, Se, VMR}).

The domains considered in this paper may be nonsmooth but should satisfy some flatness condition. Essentially, at each
boundary point and every scale, we require the boundary of the domain  be between two hyperplanes separated by a distance that
depends on the scale.    The following  defines the relevant geometry  precisely.
\begin{definition}
 Given $\delta\in (0, 1)$ and $R_0>0$, we say that $\Omega$ is a $(\delta, R_0)$-Reifenberg flat domain if for every $x\in \partial \Omega$
 and every $r\in (0, R_0]$, there exists a
 system of coordinates $\{ y_{1}, y_{2}, \dots,y_{n}\}$,
 which may depend on $r$ and $x$, so that  in this coordinate system $x=0$ and that
\[
B_{r}(0)\cap \{y_{n}> \delta r \} \subset B_{r}(0)\cap \Omega \subset B_{r}(0)\cap \{y_{n} > -\delta r\}.
\]
\end{definition}
For more on Reifenberg flat domains and their applications, we refer to the papers \cite{HM, Jon, KT1, KT2, Rei, Tor}.
We remark that Reifenberg flat domains  include $\mathcal{C}^1$ domains and  Lipschitz domains with sufficiently small Lipschitz constants (see \cite{Tor}).  Moreover, they also include certain domains with fractal boundaries and thus provide a wide range of applications.

We now recall the definition of  weighted Lorentz spaces. 
For a nonnegative locally integrable function $w$, called {\em a weight function},  the weighted Lorentz  space $L^{s,\, t}_{w}(\Om)$ with $0< s<\infty$, $0<t\leq\infty$, is the set of
measurable functions $g$ on $\Omega$ such that
\[
\|g\|_{L^{s,\, t}_{w}(\Omega)} := \left[s \int_{0}^{\infty}(\alpha^s w(\{x\in\Om: |g(x)|>\alpha\}))^{\frac{t}{s}} \frac{d\alpha}
{\alpha}\right]^{\frac{1}{t}} <+ \infty
\]
when $t\not=\infty$; for $t=\infty$ the space $L^{s,\, \infty}_{w}(\Om)$ is set to be the  Marcinkiewicz space with quasinorm
$$\|g\|_{L^{s,\, \infty}_{w}(\Omega)}:=\sup_{\alpha >0} \alpha w(\{x\in \Om: |g(x)|>\alpha\})^{\frac{1}{s}}.$$
Here for a measurable set $E\subset \RR^n$ we write  $w(E)=\int_{E} w(x) dx$.

It is easy to see that when $t=s$ the weighted Lorentz space
$L^{s,\, s}_{w}(\Om)$ is nothing but the weighted Lebesgue space $L^{s}_{w}(\Om)$, which is equivalently defined as 
\[
g\in L^{s}_{w}(\Omega) \Longleftrightarrow \norm{g}_{L^{s}_{w}(\Omega)}:= \left(\int_{\Omega}|g(x)|^s w(x)dx\right)^{\frac{1}{s}} < +\infty.
\] As usual, when $w\equiv 1$ we  simply write
$L^{s,\, t}(\Om)$ instead of $L^{s,\, t}_{w}(\Om)$.

The class of weights considered in this paper is the class of $A_\infty$ weights. Several equivalent  definitions of this class of weights can be given. For our purpose we choose the following one.

\begin{definition}\label{inversedoubling}
 We say that a weight  $w$  is an $A_{\infty}$ weight if there are two positive constants $A$ and   $\nu$ such that 
\begin{equation*}
w(E)\leq A\left ( \frac{|E|}{|B|}\right)^{\nu}w(B).
\end{equation*}
for every ball $B \subset \mathbb{R}^{n}$ and every measurable subsets $E$ of $B$. The pair $(A, \nu)$ is called the $A_\infty$ constants of $w$  
and is denoted by $[w]_{A_\infty}$. 
\end{definition}

We are now ready to state the first main result of the paper.

\begin{theorem}\label{main}
Let $2-1/n<p\leq n$, $0<q<\infty$, $0<t\leq \infty$, and let $w$ be an $ A_{\infty}$ weight.  Suppose that $\mathcal{A}$ satisfies \eqref{sublinear}-\eqref{ellipticity}. 
Then there exist constants $s=s(n, p, \alpha, \beta)>1$ and $\delta=\delta(n, p, \alpha, \beta, q, t, [w]_{A_\infty})\in (0, 1)$ such that the following holds. 
If  $[\mathcal{A}]_{s}^{R_0}\leq \delta$ and $\Om$ is $(\delta, R_0)$-Reifenberg flat for some $R_0>0$, then for any renormalized solution $u$ to 
the boundary value problem  \eqref{basicpde} we have 

\begin{equation}\label{B0bound}
\norm{\nabla u}_{L^{q, \, t}_{w}(\Omega)}\leq C \norm{{\rm\bf M}_1(\mu)^{\frac{1}{p-1}}}_{L^{q, \, t}_{w}(\Omega)}.
\end{equation}
Here  $C$ depends only on $n, p, \alpha, \beta, q, t, [w]_{A_\infty}$, and ${\rm diam}(\Om)/R_0$.
\end{theorem}

Inequality \eqref{B0bound} can be viewed as a nonlinear version of a classical result due  to Muckenhoupt and Wheeden \cite{MW} 
regarding the weighted norm equivalence of Riesz potentials and the corresponding fractional maximal functions.

We remark that from the proof of this theorem one can in fact  replace the norm  $\norm{\nabla u}_{L^{q, \, t}_{w}(\Omega)}$ 
in \eqref{B0bound} with   $\norm{{\rm\bf M}(|\nabla u|)}_{L^{q, \, t}_{w}(\Omega)}$ even when $q\leq 1$.
Hereafter, ${\rm\bf M}$ denotes the Hardy-Littlewood maximal function defined for each $f\in L^{1}_{\rm loc}(\mathbb{R}^{n})$ by
\[
{\rm\bf M}f(x) = \sup_{r> 0} \fint_{B_{r}(x)} |f(y)|dy, \qquad x\in \RR^n.
\]

\begin{remark}\label{notionofsol}
{\rm We use the notion of renormalized solutions to address equation \eqref{basicpde} in Theorem \ref{main}. Several equivalent definitions of  
renormalized solutions were given in \cite{DMOP}, two of which will be recalled later in Section \ref{app}. In fact, to obtain \eqref{B0bound}, it is enough to use 
the following  milder notion of solution. For each integer $k>0$ the truncation 
$$T_{k}(u):=\max\{-k, \min\{k,u\}\}$$ belongs to $W_0^{1,\, p}(\Om)$ and satisfies
$$-{\rm div}\,\mathcal{A}(x, \nabla T_k(u))= \mu_k$$
in the sense of distributions in $\Om$ for a finite measure $\mu_k$ in $\Om$. Moreover, if we extend both $\mu$ and $\mu_k$ by zero to $\RR^n\setminus\Om$ then
$|\mu_k|$  converges  to $|\mu|$ weakly as measures in $\RR^n$.  It is known that renormalized solutions satisfy 
these properties (see Remark \ref{quasi}).
}\end{remark}

We next state the second main result of the paper. Here we need the notion of capacity
associated to the Sobolev space $W^{1,\, s}(\RR^n)$,  $1<s<+\infty$.
For a compact set $K\subset\RR^n$,  we  define
\begin{equation*}
{\rm Cap}_{1, \, s}(K)=\inf\Big\{\int_{\RR^n}(|\nabla \varphi|^s +\varphi^s) dx: \varphi\in C^\infty_0(\RR^n),
\varphi\geq \chi_K \Big\},
\end{equation*}
where $\chi_{K}$ is the characteristic function of the set $K$.

\begin{theorem}\label{main-Ric}
Let $2-1/n<p\leq n$ and  $q>p-1$.  Assume  that $\mathcal{A}$ satisfies \eqref{sublinear}-\eqref{ellipticity}. 
Then there exist constants $s=s(n, p, \alpha, \beta)>1$ and $\delta=\delta(n, p, \alpha, \beta, q)\in (0, 1)$ such that the following holds. 
Suppose that  $[\mathcal{A}]_{s}^{R_0}\leq \delta$ and $\Om$ is $(\delta, R_0)$-Reifenberg flat for some $R_0>0$. Then there exists a constant 
$c_0=c_0(n, p, \alpha, \beta, q, {\rm diam}(\Om), {\rm diam}(\Om)/R_0)$ such that if $\om$ is a finite signed measure in $\Om$ with 
\begin{equation}\label{capcondi} 
|\om|(K\cap\Om) \leq c_0\, {\rm Cap}_{1,\, \frac{q}{q-p+1}}(K)
\end{equation}
for all compact sets $K\subset\RR^n$, then there exists a renormalized solution $u\in W_0^{1, \, q}(\Om)$ to the Riccati type equation
\begin{eqnarray}\label{Riccati}
\left\{ \begin{array}{rcl}
 -{\rm div}\,\mathcal{A}(x, \nabla u)&=& |\nabla u|^q + \om \quad \text{in} ~\Omega, \\
u&=&0  \quad \text{on}~ \partial \Omega,
\end{array}\right.
\end{eqnarray} 
such that 
\begin{equation*} 
\int_{K\cap\Om} |\nabla u|^q dx \leq C\, {\rm Cap}_{1,\, \frac{q}{q-p+1}}(K)
\end{equation*}
for all compact sets $K\subset\RR^n$, and for a constant $C>0$ depending only on $n, p, \alpha, \beta, q, {\rm diam}(\Om)$ and ${\rm diam}(\Om)/R_0$.
\end{theorem}

It is now well-known that the capacitary condition \eqref{capcondi} is sharp in the sense that if \eqref{Riccati} has a solution with $\om$ being 
 nonnegative and compactly supported in $\Om$ then  \eqref{capcondi} holds with  a different constant $c_0$ (see \cite{Ph1}). Thus Theorem \ref{main-Ric} solves 
 an open problem in \cite[page 13]{VHV} at least for compactly supported measures and  $2-1/n<p\leq n$.

It is worth mentioning that for $p=2$ sharp criteria of existence  were already obtained in the pioneering work \cite{HMV}.  The case $p=2$ and $1+2/n\leq q<2$ was also studied in \cite{GMP} for datum $\om\in L^{n(q-1)/q}(\Om)$, in which the existence of $W^{1,\, 2}_0(\Om)$ solutions was obtained.
For general $p\in (1, n]$, we refer to the recent papers \cite{Ph5} and \cite{VHV} for a brief account on Riccati type equations, especially for the cases $0<q\leq p-1$, $p-1<q< \frac{n(p-1)}{n-1}$ and $q=p$. Here we mention that   an  existence result involving condition \eqref{capcondi} in the super-natural  case $q>p$ has already been obtained in \cite{Ph1} for $\mathcal{C}^1$ domains and for  less general nonlinear structures. Thus, in particular,  Theorem \ref{main-Ric} above fills the gap in the case 
$\frac{n(p-1)}{n-1}\leq q<p$.

Moreover,   Theorem \ref{main-Ric} implies  that the condition $\om\in L^{n(q-p+1)/q, \, \infty}(\Om)$ (with a small norm) is sufficient for the solvability of \eqref{Riccati}, provided that $n(q-p+1)/q>1$, i.e., $q>\frac{n(p-1)}{n-1}$. Other sufficient conditions of Fefferman-Phong type involving Morrey spaces can also be deduced from the capacitary condition  \eqref{capcondi} (see Corollaries 3.5 and 3.6 in \cite{Ph1}). Some of those  sufficient conditions have been obtained recently in \cite{Ph5} for \eqref{Riccati} with more general domains and coefficients. It is still unclear to us whether condition \eqref{capcondi} is also sufficient for the solvability of \eqref{Riccati} under such a general assumption on $\Om$ and $\mathcal{A}$ as in \cite{Ph5}.

To conclude this section we note that Theorem \ref{main-Ric} also implies the following result regarding the size of removable singular sets for homogeneous Riccati type equations. 
For a proof see that of \cite[Theorem 3.9]{Ph1}.
\begin{theorem}\label{remove-S-nec}
Let $2-1/n<p\leq n$ and  $q>p-1$.  Assume  that $\mathcal{A}$ satisfies \eqref{sublinear}-\eqref{ellipticity}. 
Then there exist constants $s=s(n, p, \alpha, \beta)>1$ and $\delta=\delta(n, p, \alpha, \beta, q)\in (0, 1)$ such that the following holds. 
Suppose that  $[\mathcal{A}]_{s}^{R_0}\leq \delta$ and $\Om$ is $(\delta, R_0)$-Reifenberg flat for some $R_0>0$. 
Let $K$  be a compact set in  $\Om$ such that  any solution $u$  to the problem
\begin{equation*}
\left\{\begin{array}{l}
 u\in W^{1,\, q}_{\rm loc}(\Om\setminus K), ~ and\\
-{\rm div} \mathcal{A}(x, \nabla u) =|\nabla u|^q ~ in~
 \mathcal{D}'(\Om\setminus K)
\end{array}
\right.
\end{equation*}
is also a solution to a similar problem with $\Om$ in place of $\Om\setminus K$. Then there holds  
$${\rm Cap}_{1,\,\frac{q}{q-p+1}}(K)=0.$$ 
\end{theorem}

We would like to mention that the proof of Theorem \ref{remove-S-nec} relies on the existence of the so-called capacitary measure $\mu^K$ for a compact set 
$K$ with ${\rm Cap}_{1,\, \frac{q}{q-p+1}}(K)>0$. It is well-known that ${\rm supp} (\mu^K) \subset K$, $\mu^K(K)={\rm Cap}_{1,\, \frac{q}{q-p+1}}(K)$, and moreover $$\mu^K(E) \leq {\rm Cap}_{1,\, \frac{q}{q-p+1}}(E)$$
for all compact  (and Borel) sets $E$. Here the last property follows, e.g., from  \cite[Proposition 6.3.13]{AH} and the capacitability of Borel sets.


\section{Unweighted local interior and boundary estimates}\label{sec4}

In this section,  we   obtain  certain  local interior and boundary comparison  estimates that are essential to our development
later. First let us consider the interior ones. With $u\in W^{1,\, p}_{\rm loc}(\Om)$,
for each ball $B_{2R}=B_{2R}(x_0)\Subset\Om$ we defined $w\in u+W_{0}^{1,\,p}(B_{2R})$ as the unique solution to the Dirichlet problem
\begin{equation}\label{firstapprox}
\left\{ \begin{array}{rcl}
 \text{div}~ \mathcal{A}(x, \nabla w)&=&0   \quad \text{in} ~B_{2R}, \\ 
w&=&u  \quad \text{on}~ \partial B_{2R}.
\end{array}\right.
\end{equation}

Then a well-known version of Gehring's lemma applied to the function $w$ defined above  yields the following result (see \cite[Theorem 6.7]{Giu} and
\cite[Remark 6.12]{Giu}).
\begin{lemma}\label{RVH} With $u\in W^{1,\, p}_{\rm loc}(\Om)$, let $w$ be as in \eqref{firstapprox}. 
Then there exists a constant $\theta_0=\theta_0(n,p,\alpha,\beta)>1$ such that for any $t\in (0,p]$ the reverse H\"older type inequality
\begin{equation*}
\left(\fint_{B_{\rho/2}(z)} |\nabla w|^{\theta_0 p} dx\right)^{\frac{1}{\theta_0 p}} \leq C \left(\fint_{B_{\rho}(z)} |\nabla w|^t dx\right)^{\frac{1}{t}}
\end{equation*} 
holds for all balls $B_{\rho}(z)\subset B_{2R}(x_0)$ for a constant $C$ depending only on $n,p,\alpha, \beta, t$. 
\end{lemma}

The following important comparison lemma involving an estimate ``below the natural growth exponent" was   established in \cite{Mi2} in the case
$p\geq 2$ and in  \cite{DM2} in the case $2-1/n<p<2$.

Hereafter, for the sake of convenience, given a function $u$, a measure $\mu$, and a ball $B_r(x)$ we define  
\begin{equation*}
F(p, \nabla u,\mu, B_r(x))= \left[ \frac{|\mu|(B_{r}(x))}{r^{n-1}}\right]^{\frac{1}{p-1}} +  \left[ \frac{|\mu|(B_{r}(x))}{r^{n-1}}\right]\left(\fint_{B_{r}(x)}|\nabla u|dx\right)^{2-p}
\end{equation*} 
if $ 2- 1/n <p <2$, and
$$F(p, \nabla u,\mu, B_r(x))=  \left[ \frac{|\mu|(B_{r}(x))}{r^{n-1}}\right]^{\frac{1}{p-1}}$$
if $p\geq 2$.

\begin{lemma} \label{DM}
With $p>2-1/n$, let $u\in W^{1,\, p}_{\rm loc}(\Om)$ be a solution of \eqref{basicpde} and let $w$ be as in \eqref{firstapprox}. Then there is a constant $C=C(n,p,\alpha,\beta)$ such that 
\begin{eqnarray*}
\fint_{B_{2R}} |\nabla u-\nabla w|dx&\leq& C F(p, \nabla u,\mu, B_{2R}(x_0)),
\end{eqnarray*}
and thus
\begin{eqnarray*}
\fint_{B_{2R}} |\nabla w|dx&\leq&  C\fint_{B_{2R}} |\nabla u|dx + C F(p, \nabla u,\mu, B_{2R}(x_0)).
\end{eqnarray*}
\end{lemma}

Next with $u$ and $w$ being as in \eqref{firstapprox} where $B_{2R}=B_{2R}(x_0)$ we further define another function $v\in w+ W_{0}^{1,\,p}(B_{R})$
as the unique solution to the Dirichlet problem
\begin{equation}\label{secondapprox}
\left\{ \begin{array}{rcl}
 \text{div}~ \overline{\mathcal{A}}_{B_{R}}(\nabla v)&=&0   \quad \text{in} ~B_{R}, \\ 
v&=&w  \quad \text{on}~ \partial B_{R},
\end{array}\right.
\end{equation}
where $B_R=B_R(x_0)$.

We shall omit the proof of the next lemma as it is completely similar to the proof of Lemma 3.9 (for  $p\geq 2$) and Lemma 3.10 (for $1<p<2$) in \cite{Ph3}.  
\begin{lemma}\label{BMOapprox1}
Let $s=\frac{\theta p}{(\theta-1)(p-1)}$ for some $\theta\in(1,\theta_0]$, where $\theta_0>1$ is as in Lemma \ref{RVH}. Let also 
$w$ and $v$ be as in \eqref{firstapprox} and \eqref{secondapprox}, respectively. Then there is a constant $C=C(n,p,\alpha,\beta)$ such that 
\begin{equation*}
\fint_{B_{R}} |\nabla v-\nabla w|^pdx \leq C \left(\fint_{B_{R}}\Upsilon(\mathcal{A},
B_{R})^{s}dx\right)^{\min\{p, \frac{p}{p-1}\}/s} \left(\fint_{B_{2R}}|\nabla w| dx\right)^{p}. 
\end{equation*} 
\end{lemma}

\begin{corollary} \label{mainlocalestimate-interior} Let $s=\frac{\theta p}{(\theta-1)(p-1)}$ for some $\theta\in(1,\theta_0]$, where $\theta_0>1$ is as found in Lemma \ref{RVH}.  If $u\in W_{\rm loc}^{1,\, p}(\Om)$ is a weak solution of 
\eqref{basicpde} in $B_{2R}=B_{2R}(x_0)\Subset\Om$ with $2R\leq R_0$, 
then there is a function $v\in W^{1,\, p}(B_R)\cap W^{1,\, \infty}(B_{R/2})$ such that 
\begin{equation}\label{Linftyboundint}
\norm{\nabla v}_{L^{\infty}(B_{R/2})}\leq C \fint_{B_{2R}}|\nabla u| dx + C F(p, \nabla u,\mu, B_{2R}(x_0)),
\end{equation}
and 
\begin{eqnarray}\label{u-vv}
&&\fint_{B_{R}} |\nabla u-\nabla v|dx\leq C F(p, \nabla u,\mu, B_{2R}(x_0)) +\\
 && \quad +\, C ([\mathcal{A}]^{R_0}_{s})^{\min\{1, \frac{1}{p-1}\}}\left( \fint_{B_{2R}}|\nabla u|dx + 
 F(p, \nabla u,\mu, B_{2R}(x_0)) \right).   \nonumber
\end{eqnarray}
Here $C=C(n,p,\alpha, \beta)$.
\end{corollary}

\begin{proof} Let $w$ and $v$ be as in \eqref{firstapprox} and \eqref{secondapprox}. 
By standard interior regularity (see, e.g., Proposition 3.3 in \cite{DiB}) and Lemma \ref{RVH} we have 
\begin{eqnarray*}
\norm{\nabla v}_{L^{\infty}(B_{R/2})}&\leq& C_1 \left(\fint_{B_{R}}|\nabla v|^p dx\right)^{1/p}\\
&\leq& C_2\left(\fint_{B_{R}}|\nabla w|^p dx\right)^{1/p} \\
&\leq& C_3\fint_{B_{2R}}|\nabla w| dx.
\end{eqnarray*}
 
 Thus it follows from Lemma \ref{DM}  that the  bound \eqref{Linftyboundint} holds.

On the other hand, by Lemma \ref{BMOapprox1} and H\"older's inequality
we find  
\begin{equation*}
\fint_{B_{R}} |\nabla v-\nabla w| dx \leq C ([\mathcal{A}]^{R_0}_{s})^{\min\{1, \frac{1}{p-1}\}} \fint_{B_{2R}}|\nabla w| dx. 
\end{equation*} 

This and triangle inequality yield
\begin{equation*}
\fint_{B_{R}} |\nabla u-\nabla v| dx \leq  \fint_{B_{R}} |\nabla u-\nabla w| dx +
  C ([\mathcal{A}]^{R_0}_{s})^{\min\{1, \frac{1}{p-1}\}} \fint_{B_{2R}}|\nabla w| dx.
\end{equation*} 

Thus inequality  \eqref{u-vv}  then  follows from Lemma \ref{DM}.
\end{proof}

Next, we consider  the corresponding estimates near the boundary.  Recall that $\Om$ is a $(\delta, R_0)$-Reifenberg flat domain.
Fix $x_0\in\partial\Om$ and $0<R<R_0/10$. With $u\in W^{1,\, p}_0(\Om)$  being a solution to \eqref{basicpde} we define $w\in u+W_{0}^{1,\,p}(\Om_{10R}(x_0))$ as the unique solution to the Dirichlet problem
\begin{equation}\label{wapprox}
\left\{ \begin{array}{rcl}
 \text{div}~ \mathcal{A}(x, \nabla w)&=&0   \quad \text{in} ~\Om_{10R}(x_0), \\ 
w&=&u  \quad \text{on}~ \partial \Om_{10R}(x_0).
\end{array}\right.
\end{equation}
Hereafter, the notation $\Om_r(x)$ indicates the set $\Om\cap B_r(x)$.

We next extend $\mu$ and $u$ by zero to $\RR^n\setminus\Om$ and then extend $w$ by $u$ to $\RR^n\setminus\Om_{10R}(x_0)$.
Note  that for a $(\delta, R_0)$-Reifenberg flat domain we have a density estimate
$$|B_{t}(x)\cap (\RR^n\setminus\Om)|\geq c |B_{t}(x)|, \quad \forall x\in\partial\Om,~0<t<R_0,$$ 
which holds with  $c=[(1-\delta)/2]^n\geq 4^{-n}$ provided $\delta<1/2$. Thus by \cite[Lemma 2.5]{Ph4} we have the following boundary counterpart of Lemma \ref{RVH}.

\begin{lemma}\label{RVHbdry} With $u\in W^{1,\, p}_{0}(\Om)$,  let $w$ be as in  \eqref{wapprox} and suppose that $\delta<1/2$.
Then there exists a constant $\theta_0=\theta_0(n,p,\alpha,\beta)>1$ such that for every $t\in (0,p]$ the reverse H\"older type inequality
\begin{equation*}
\left(\fint_{B_{\rho/2}(z)} |\nabla w|^{\theta_0 p} dx\right)^{\frac{1}{\theta_0 p}} \leq C \left(\fint_{B_{3\rho}(z)} |\nabla w|^t dx\right)^{\frac{1}{t}}
\end{equation*} 
holds for all balls $B_{3\rho}(z)\subset B_{10R}(x_0)$ with a constant $C=C(n,p, t, \alpha, \beta)$. 
\end{lemma}

The following lemma was also  obtained in \cite{Ph4}.
\begin{lemma} \label{DMbdry}
With $p>2-1/n$, let $u\in W^{1,\, p}_{0}(\Om)$ be a solution of \eqref{basicpde} and let $w$ be as in \eqref{wapprox}. Then there is a constant $C=C(n,p,\alpha,\beta)$ such that 
\begin{equation*}
\fint_{B_{10R}(x_0)} |\nabla u-\nabla w|dx \leq C F(p, \nabla u,\mu, B_{10R}(x_0)),
\end{equation*}
and thus 
\begin{equation*}
\fint_{B_{10R}(x_0)} |\nabla w|dx \leq C \fint_{B_{10R}(x_0)} |\nabla u|dx + C F(p, \nabla u,\mu, B_{10R}(x_0)).
\end{equation*}
\end{lemma}

With  $x_0\in\partial\Om$ and $0<R<R_0/10$ as above, we now set
$\rho= R(1-\delta)$ so that  $0<\rho/(1-\delta)<R_0/10$. 
From the definition of Reifenberg flat domains we deduce  that 
there exists a coordinate system  $\{z_{1}, z_{2}, \cdots, z_{n}\}$
with the origin $0 \in \Omega$ such that in this coordinate system $x_0=(0, \dots, 0, -\rho  \delta/(1-\delta))$ and
\[
B_{\rho}^{+}(0) \subset \Omega\cap B_{\rho}(0) \subset B_{\rho}(0)\cap \{z=(z_1, z_2, \dots, z_n):z_{n} > -2\rho \delta/(1-\delta)\}.
\]
Thus if $\delta <1/2$ we have 
\[
B_{\rho}^{+}(0) \subset \Omega\cap B_{\rho}(0) \subset B_{\rho}(0)\cap \{z=(z_1, z_2, \dots, z_n):z_{n} > -4\delta\rho\}.
\]
Here $B^+_{\rho}(0):= B_{\rho}(0) \cap \{z=(z_1, z_2, \dots, z_n):z_{n} >0\}$ denotes an upper half ball in the corresponding coordinate system.

We next define another function $v\in w+ W_{0}^{1,\,p}(\Om_{\rho}(0))$
as the unique solution to the Dirichlet problem
\begin{equation}\label{vapprox}
\left\{ \begin{array}{rcl}
 \text{div}~ \overline{\mathcal{A}}_{B_{\rho}(0)}(\nabla v)&=&0   \quad \text{in} ~\Om_{\rho}(0), \\ 
v&=&w  \quad \text{on}~ \partial \Om_{\rho}(0).
\end{array}\right.
\end{equation}

Then we set $v$ to be equal to $w$ outside  $\Om_{\rho}(0)$.
Similar to Lemma \ref{BMOapprox1} we have the following one.

\begin{lemma}\label{BMOapprox2}
Let $s=\frac{\theta p}{(\theta-1)(p-1)}$ for some $\theta\in(1,\theta_0]$, where $\theta_0>1$ is as in Lemma \ref{RVHbdry}. Let also 
$w$ and $v$ be as in \eqref{wapprox} and \eqref{vapprox}, respectively. Then there is a constant $C=C(n,p,\alpha,\beta)$ such that 
\begin{eqnarray*}
\lefteqn{\fint_{B_{\rho}(0)} |\nabla v-\nabla w|^pdx}\\
&& \leq C \left(\fint_{B_{\rho}(0)}\Upsilon(\mathcal{A}, B_{\rho}(0))(x)^{s}dx\right)^{\min\{p, \frac{p}{p-1}\}/s} \left(\fint_{B_{6\rho}(0)}|\nabla w| dx\right)^p. 
\end{eqnarray*}
\end{lemma}

By standard interior regularity results  we see that the function $v$ in Lemma \ref{BMOapprox2} belongs to $W^{1, \, \infty}_{\rm loc}(\Om_{\rho}(0))$. However, as the boundary of $\Om$ can be very irregular  the $L^\infty$-norm of $\nabla v$ up to $\partial\Om\cap B_{\rho}(0)$ could be unbounded. For our purpose we need to consider another intermediate equation: 
\begin{eqnarray}\label{perturbedpdeboundary}
\left\{ \begin{array}{rcl}
 \text{div}~ \overline{\mathcal{A}}_{B_{\rho}(0)}(\nabla V)&=0 \quad \text{in} ~B^{+}_{\rho}(0), \\
V &= 0 \quad\text { on $T_{\rho}$},
\end{array}\right.
\end{eqnarray}
where $T_{\rho}$ is the flat portion of $\partial B^{+}_{\rho}(0)$.

By definition a function $V\in W^{1,\, p}(B^{+}_{\rho}(0))$ is a weak solution of \eqref{perturbedpdeboundary} if its zero extension to $B_{\rho}(0)$ belongs to 
 $W^{1,\, p}(B_{\rho}(0))$ and if $\text{div}~ \overline{\mathcal{A}}_{B_{\rho}(0)}(\nabla V)=0$ weakly in $B^{+}_{\rho}(0)$, i.e.,
\[
\int_{B^{+}_{\rho}(0)} \overline{\mathcal{A}}_{B_{\rho}(0)}(\nabla V)\cdot \nabla \varphi dx = 0
\]
for all $\varphi \in W_{0}^{1,\, p} (B^{+}_{\rho}(0))$.

We  have the following gradient $L^{\infty}$ estimate up to the boundary for $V$.
\begin{lemma}[\cite{L}] \label{w1infty}
For any weak solution $V \in W^{1,\, p}(B^{+}_{\rho}(0))$ of (\ref{perturbedpdeboundary}), 
we have
\[
\|\nabla V\|_{L^{\infty}(B^+_{\rho'/2}(0))}^{p} \leq C\fint_{B^+_{\rho'}(0)} |\nabla V|^{p}dx, \qquad 0<\rho'\leq \rho,
\]
for some universal constant $C >0$. Moreover, $\nabla V$ is  continuous up to $T_{\rho}$. 
\end{lemma}
We also need the following result concerning the zero extensions of solutions of  (\ref{perturbedpdeboundary}) to $B_\rho(0)$. 

\begin{lemma} \label{Vextension}
If $V \in W^{1,\, p}(B^{+}_{\rho}(0))$ is a weak solution of  (\ref{perturbedpdeboundary}), then its zero extension from $B^+_\rho(0)$ to 
$B_\rho(0)$ solves
\begin{equation*}
 {\rm div}~ \overline{\mathcal{A}}_{B_{\rho}(0)}(\nabla V)= \frac{\partial f}{\partial x_n}, 
\end{equation*}
weakly in $B_{\rho}(0)$, where for  $x=(x', x_n)\in B_\rho(0)$,
\begin{equation*}
\overline{\mathcal{A}}_{B_{\rho}}=\left(\overline{\mathcal{A}}^{1}_{B_{\rho}}, \dots, \overline{\mathcal{A}}^{n}_{B_{\rho}}\right), \quad  and~ 
f(x)=-\chi_{\{x_n<0\}} \overline{\mathcal{A}}^{n}_{B_{\rho}}\left(\nabla V(x', 0)\right).
\end{equation*} 
\end{lemma}
\begin{proof}
Let $h$ be a $C^\infty$ function on the real line such that $h=0$ on $(-\infty, 1/2)$ and $h=1$ on $(1, \infty)$.  
Set $B_\rho=B_\rho(0)$. Then for any $\varphi\in C^\infty_0(B_\rho)$ and $\epsilon>0$ we see that the function $x=(x', x_n)\rightarrow h(x_n/\epsilon)\varphi(x)$ belongs to $C^\infty_0(B^{+}_\rho)$.
Thus we get
$$\int_{B^{+}_\rho}\overline{\mathcal{A}}_{B_{\rho}}(\nabla V)\cdot\nabla [h(x_n/\epsilon)\varphi(x)]dx=0,$$
which yields
\begin{eqnarray*}
\lefteqn{\int_{B^{+}_\rho}\overline{\mathcal{A}}_{B_{\rho}}(\nabla V)\cdot \nabla \varphi  h(x_n/\epsilon) dx}\\
&=& -\int_{B^{+}_\rho}\overline{\mathcal{A}}^{n}_{B_{\rho}}(\nabla V) h'(x_n/\epsilon)\epsilon^{-1} \varphi dx\\
&=& -\int_{B^+_\rho}\overline{\mathcal{A}}^{n}_{B_{\rho}}(\nabla V(x', x_n)) h'(x_n/\epsilon) \epsilon^{-1} \varphi(x',x_n) dx' dx_n \\
&=& - \int_{0}^{\rho} \int_{|x'|< \sqrt{\rho^2-x_n^2}}\overline{\mathcal{A}}^{n}_{B_{\rho}}(\nabla V(x', x_n))  \varphi(x',x_n) dx'\, h'(x_n/\epsilon) \epsilon^{-1} dx_n.
\end{eqnarray*}

Now letting $\epsilon\rightarrow 0^+$   we obtain
\begin{eqnarray*}
\int_{B^+_\rho}\overline{\mathcal{A}}_{B_{\rho}}(\nabla V)\cdot \nabla \varphi dx &=& -\int_{\{|x'|<\rho\}}
\overline{\mathcal{A}}^{n}_{B_{\rho}}(\nabla V(x', 0))  \varphi(x',0) dx'\\
&=&-\int_{B^-_\rho} \overline{\mathcal{A}}^{n}_{B_{\rho}}(\nabla V(x', 0))  \frac{\partial \varphi}{\partial x_n}(x) dx,
\end{eqnarray*}
where the last equality follows from integrating by parts.
This gives the desired result since by \eqref{sublinear} and the fact that $\nabla V=0$ outside $B^+_6$ it holds that 
$$\int_{B^+_\rho}\overline{\mathcal{A}}_{B_{\rho}}(\nabla V)\cdot \nabla \varphi dx=\int_{B_\rho}\overline{\mathcal{A}}_{B_{\rho}}(\nabla V)\cdot \nabla \varphi dx.$$
\end{proof}

We now consider a scaled version of equation \eqref{vapprox}:
\begin{equation}\label{scaledvto1}
\left\{ \begin{array}{rcl}
 {\rm div}~ \overline{\mathcal{A}}_{B_{1}(0)}(\nabla v)&=&0   \quad \text{in} ~\Om_{1}(0), \\ 
v&=&0  \quad \text{on}~ \partial\Om\cap B_{1}(0),
\end{array}\right.
\end{equation}
under the geometric setting
\begin{equation}\label{geometric1}
 B_{1}^{+}(0)\subset \Om_1(0) \subset B_{1}(0)\cap \{ x_{n} > -4 \delta \}.
\end{equation}

As above a function $v\in W^{1,\, p}(\Om_{1}(0))$ is a solution of \eqref{scaledvto1} if its zero extension to $B_1(0)$ belongs to $v\in W^{1,\, p}(B_{1}(0))$ and 
${\rm div}~ \overline{\mathcal{A}}_{B_{1}(0)}(\nabla v)=0$ weakly in $\Om_{1}(0)$.

We now continue with the following important perturbation result. Earlier results of this kind can be found in \cite{BW1, BW2}.

\begin{lemma}\label{comparisonboundary0} 
Suppose that $\mathcal{A}: \RR^n\times\RR^n\rightarrow \RR^n$ satisfies \eqref{sublinear}-\eqref{monotone}.
For any $\epsilon > 0$ there exists a small $\delta = \delta(n, p, \alpha, \beta, \epsilon) > 0$ such that  if $v \in W^{1,\, p}(\Om_{1}(0))$ is a  solution of \eqref{scaledvto1}  along with
the geometric setting  \eqref{geometric1} and the bound
$$ \fint_{B_{1}(0)} |\nabla v|^{p}dx \leq 1, $$
then there exists a weak solution $V\in W^{1,\, p}(B^{+}_{1}(0))$ of (\ref{perturbedpdeboundary}) with $\rho=1$, whose  zero extension to $B_{1}(0)$  satisfies
 \[\fint_{B_{1}(0)}| v -  V|^{p}dx \leq \epsilon ^{p}.
\]
\end{lemma}
\begin{proof} As in \cite{BW1, BW2}, we shall argue by contradiction. To this end, suppose that the conclusion were false. Then there would exist an $\epsilon_0>0$, a sequence of nonlinearities $\{\mathcal{A}_{k}\}_{k=1}^{\infty}$ satisfying
\eqref{sublinear}-\eqref{monotone}, a sequence of domains $\{\Om^k\}_{k=1}^{\infty}$,
and  a sequence of functions $\{v_k\}_{k=1}^{\infty}\subset W^{1,\, p}(\Om^{k}_1)$  such that
\begin{equation}\label{k-Ball-1}
B_{1}^{+}\subset \Omega^{k}_{ 1}\subset B_{1}\cap \{ x_{n} > -1/2k \},
\end{equation}
\begin{equation}\label{k-approx}
 \text{div}~ \overline{\mathcal{A}_{k}}_{B_{1}}(\nabla v_k)=0   \quad \text{in} ~\Om^{k}_{1}, 
\end{equation}
and the zero extension of each $v_k$ to $B_1$ satisfies
\begin{equation}\label{k-w-1}
\fint_{B_{1}} |\nabla v_k|^{p}dx \leq 1,
\end{equation}
but
\begin{equation}\label{k-contra}
\fint_{B_1}| v_k -  V_k|^{p}dx \geq \epsilon_{0}^{p}
\end{equation}
for any weak solution $V_k$ of 
\begin{eqnarray}\label{k-flatV}
\left\{ \begin{array}{rcl}
 \text{div}~ \overline{\mathcal{A}_{k}}_{B_{1}}(\nabla V_k)&=0 \quad \text{in} ~B^{+}_{1}, \\
V_k &= 0 \quad\text { on $T_{1}$}.
\end{array}\right.
\end{eqnarray}

By \eqref{k-Ball-1},  \eqref{k-w-1} and Poincar\'e's inequality  it follows that 
$$\norm{v_k}_{W^{1,\, p}(B_{1})}\leq C\norm{\nabla v_k}_{L^p(B_1)}\leq C.$$

Thus there exist $v_0\in W^{1,\, p}(B_{1})$ and a subsequence, still denoted by $\{v_k\}$, such that
\begin{equation}\label{strongL}
v_k \rightarrow v_0 \quad {\rm weakly~in~}W^{1,\, p}(B_{1}),
~{\rm and}~ {\rm strongly ~in~}L^{p}(B_{1}).
\end{equation}
Moreover, as $v_k= 0$ on $B_1 \setminus \Om^{k}_1$  it follows from \eqref{k-Ball-1} that $v_0=0$ in $B^{-}_{1}:=\{x\in B_1: x_n < 0\}$.

By \eqref{sublinear} and the second inequality in \eqref{ellipticity} the sequence $\{\overline{\mathcal{A}_{k}}_{B_1}\}$ is locally uniformly Lipschitz in $\RR^n\setminus\{0\}$ 
and thus by taking another subsequence if necessary we see that 
$\{\overline{\mathcal{A}_{k}}_{B_1}\}$ converges  locally uniformly in $\RR^n\setminus\{0\}$ to a nonlinearity 
$\mathcal{A}_{0}$ that also satisfies \eqref{sublinear} and \eqref{monotone}. Moreover, we may also define $\mathcal{A}_{0}(0)=0\in \RR^n$ without violating 
\eqref{sublinear} and \eqref{monotone}.

We next claim that $v_0$ solves  
\begin{equation*}
\left\{ \begin{array}{rcl}
 \text{div}~ \mathcal{A}_{0}(\nabla v_0)&=&0   \quad \text{in} ~B^{+}_{1}, \\ 
v_0&=&0  \quad \text{on}~ T_{1}.
\end{array}\right.
\end{equation*}

To verify this  we  observe that by \eqref{k-approx} and \eqref{k-Ball-1} one has
\begin{equation}\label{from-k}
\int_{B^+_1} \overline{\mathcal{A}_k}_{B_{1}}(\nabla v_k)\cdot \nabla \varphi dx=0
\end{equation}
for all $\varphi\in W_{0}^{1,\,p}(B^{+}_1)$. Moreover, in view of \eqref{sublinear} we have
$$\int_{B^+_1}|\overline{\mathcal{A}_{k}}_{B_{1}}(\nabla v_k)|^{\frac{p}{p-1}} dx \leq C \int_{B^+_1}|\nabla v_k|^p dx \leq C.$$

Thus there exists a subsequence, still denoted by $\{\overline{\mathcal{A}_k}_{B_{1}}(\nabla v_k)\}$, and a vector field
${\bf A}$ belonging to  $L^{\frac{p}{p-1}}(B^+_1, \RR^n)$ such that 
\begin{equation}\label{weakL}
\overline{\mathcal{A}_k}_{B_{1}}(\nabla v_k)\rightarrow {\bf A}
\quad {\rm weakly~ in}~ L^{\frac{p}{p-1}}(B^+_1, \RR^n). 
\end{equation}

This and \eqref{from-k} give 
\begin{equation}\label{A-equa}
\int_{B^+_1} {\bf A}\cdot \nabla \varphi dx=0
\end{equation}
for all $\varphi\in W_{0}^{1,\,p}(B^{+}_1)$. Therefore it is enough to show that 
\begin{equation}\label{desire}
\int_{B^+_1} [\mathcal{A}_{0}(\nabla v_0)-{\bf A}]\cdot \nabla \varphi dx =0, \qquad \forall \varphi\in W_{0}^{1,\,p}(B^{+}_1).
\end{equation}

For this goal we fix a function $g\in W^{1,\, p}(B^+_1)$ and any $\phi\in C^{\infty}_{0}(B^{+}_{1})$ such that $\phi\geq 0$. By \eqref{monotone} we have 
\begin{eqnarray*}
0&\leq& \int_{B^+_1}\phi [\overline{\mathcal{A}_k}_{B_{1}}(\nabla v_k)-\overline{\mathcal{A}_k}_{B_{1}}(\nabla g)]\cdot (\nabla v_k-\nabla g) dx\\
&=& \int_{B^+_1}\phi \overline{\mathcal{A}_k}_{B_{1}}(\nabla v_k)\cdot \nabla v_k dx- \int_{B^+_1}\phi \overline{\mathcal{A}_k}_{B_{1}}(\nabla v_k)\cdot \nabla g dx\\
&& - \int_{B^+_1}\phi \overline{\mathcal{A}_k}_{B_{1}}(\nabla g)\cdot (\nabla v_k-\nabla g) dx\\
&=:& I_{1}-I_2-I_3.
\end{eqnarray*}

In view of \eqref{from-k} we have 
$$I_1= \int_{B^+_1}\phi \overline{\mathcal{A}_k}_{B_{1}}(\nabla v_k)\cdot \nabla v_k dx=- \int_{B^+_1} v_k \overline{\mathcal{A}_k}_{B_{1}}(\nabla v_k)\cdot \nabla \phi dx,$$
which by \eqref{strongL}, \eqref{weakL} and \eqref{A-equa} yields
$$I_1\rightarrow - \int_{B^+_1} v_0 {\bf A}\cdot \nabla \phi dx= \int_{B^+_1} \phi {\bf A}\cdot \nabla v_0 dx, \quad {\rm as}~ k\rightarrow\infty.$$

Also, using \eqref{weakL} we find
$$I_2 \rightarrow \int_{B^+_1} \phi {\bf A}\cdot \nabla g dx.$$

To treat the term $I_3$ we observe that the sequence $\overline{\mathcal{A}_k}_{B_{1}}(\nabla g)$ converges to $\mathcal{A}_{0}(\nabla g)$ pointwise a.e. in $B^+_1$ and 
hence strongly in $L^{\frac{p}{p-1}}(B^+_1)$ by \eqref{sublinear} and Lebesgue's Dominated Convergence Theorem.  Thus it follows from  \eqref{strongL} that 
$$\quad I_3 \rightarrow \int_{B^+_1}\phi \mathcal{A}_{0}(\nabla g)\cdot (\nabla v_0-\nabla g) dx.$$

Now the convergences of $I_1, I_2$ and $I_3$ imply that 
$$\int_{B^+_1}\phi [{\bf A}-\mathcal{A}_{0}(\nabla g)]\cdot (\nabla v_0-\nabla g) dx\geq 0,$$
for all $\phi\in C^{\infty}_{0}(B^{+}_1)$, $\phi\geq 0$, and $g\in W^{1, \, p}(B^{+}_1)$. We now fix 
$\varphi\in W^{1,\,p}_{0}(B^{+}_1)$ and set $g=v_0-\gamma \varphi$ where $\gamma>0$ in the above inequality to obtain
$$\int_{B^+_1}\phi [{\bf A}-\mathcal{A}_{0}(\nabla v_0-\gamma \nabla \varphi)]\cdot \nabla \varphi dx\geq 0.$$

Letting $\gamma \rightarrow 0$ yields
$$\int_{B^+_6}\phi [{\bf A}-\mathcal{A}_{0}(\nabla v_0)]\cdot \nabla \varphi dx\geq 0,$$
which by replacing $\varphi\in W^{1,\,p}_{0}(B^{+}_1)$ with $-\varphi$ we actually get
$$\int_{B^+_1}\phi [{\bf A}-\mathcal{A}_{0}(\nabla v_0)]\cdot \nabla \varphi dx=0.$$

Since this holds for all for all $\phi\in C^{\infty}_{0}(B^{+}_1)$, $\phi\geq 0$, we obtain \eqref{desire} as desired.

Finally, to find a contradiction  we take in \eqref{k-contra} $V_k$ to be the unique solution of  \eqref{k-flatV} such that $V_k-v_0 \in W_{0}^{1,\, p}(B^+_1)$. As
\begin{equation*}
\int_{B^+_{1}} |\nabla V_k|^p dx  \leq C \int_{B^+_{1}} |\nabla v_0|^p dx\leq C, \qquad \forall k\geq 1,
\end{equation*}
by passing to a subsequence we may assume that 
\begin{equation}\label{strongLV}
V_k \rightarrow V_0 \quad {\rm weakly~in~}W^{1,\, p}(B^+_{1}),
~{\rm and}~ {\rm strongly ~in~}L^{p}(B^+_{1})
\end{equation}
for some function $V_0\in v_0+ W^{1,\,p}(B^+_1)$. Now arguing as in the case of the sequence $\{v_k\}$ above we find that $V_0$ must be a solution of 
\begin{equation*}
\left\{ \begin{array}{rcl}
 \text{div}~ \mathcal{A}_{0}(\nabla V_0)&=&0   \quad \text{in} ~B^{+}_{1}, \\ 
V_0&=&0  \quad \text{on}~ T_{1}.
\end{array}\right.
\end{equation*}

Thus by uniqueness, which is available here since $\mathcal{A}_0$ satisfies \eqref{monotone}, we get $V_0=v_0$. By \eqref{strongL} and \eqref{strongLV} and sending $k\rightarrow \infty$ in \eqref{k-contra}  we reach a contradiction.
\end{proof}

\begin{lemma}\label{corB6+}
Suppose that $\mathcal{A}: \RR^n\times\RR^n\rightarrow \RR^n$ satisfies \eqref{sublinear}-\eqref{ellipticity}.
For any $\epsilon > 0$ there exists a small $\delta= \delta(n, p, \alpha, \beta, \epsilon) > 0$ such that  if $v \in W^{1,\, p}(\Om_{1}(0))$ is a  solution of \eqref{scaledvto1}  along with
the geometric setting  \eqref{geometric1} and the bound
\begin{equation}\label{uniboundw}
\fint_{B_{1}(0)} |\nabla v|^{p}dx \leq 1,
\end{equation}
then there exists a weak solution $V\in W^{1,\, p}(B^{+}_{1}(0))$ of (\ref{perturbedpdeboundary}) with $\rho=1$, whose  zero extension to $B_{1}(0)$  satisfies
\begin{equation}\label{LinftyboundforV}
\norm{\nabla V}_{L^{\infty}(B_{1/4}(0))} \leq C
\end{equation} 
and
 \[\fint_{B_{1/8}(0)}|\nabla v - \nabla V|^{p}dx \leq \epsilon ^{p}.\]
Here $C=C(n, p, \alpha, \beta)$.
\end{lemma}

\begin{proof}
Given $\eta\in (0, 1)$ by applying Lemma \ref{comparisonboundary0} we find a small $\delta=\delta(n, p, \alpha, \beta, \eta)>0$ and a weak solution $V\in W^{1,\, p}(B^{+}_{1})$ of \eqref{perturbedpdeboundary}, with $\rho=1$, such that 
\begin{equation} \label{eta-ap}
\int_{B_1} |v-V|^p \leq \eta^p.
\end{equation}

Using $\phi^p V$ with $\phi\in C^{\infty}_{0}(B_1)$, $0\leq \phi\leq 1$, and $\phi\equiv 1$ on $B_{1/2}$, as a test function
in \eqref{perturbedpdeboundary} we  get the following Caccioppoli's inequality
$$\int_{B_{1/2}}|\nabla V|^p dx \leq C \int_{B_1} |V|^p dx. $$

This yields 
\begin{eqnarray*}
\int_{B_{1/2}} |\nabla V|^p dx &\leq& C \int_{B_1} \left(|v-V|^p +|v|^p \right)dx  \\
&\leq& C \int_{B_1} \left(|v-V|^p +|\nabla v|^p \right)dx\\
&\leq& C(\eta^p+ 1)\leq C
\end{eqnarray*}
by Poincar\'e's inequality,   \eqref{uniboundw} and \eqref{eta-ap}. Thus using Lemma \ref{w1infty} we obtain \eqref{LinftyboundforV}.

By Lemma \ref{Vextension} the zero extension of $V$ from $B^+_1$ to $B_1$  satisfies
\begin{equation}\label{EquV}
\text{div}\, \overline{\mathcal{A}}_{B_{1}}(\nabla V)= \frac{\partial f}{\partial x_n} 
\end{equation}
weakly in $B_1$, where $f(x)=-\chi_{\{x_n<0\}} \overline{\mathcal{A}}^{n}_{B_{1}}\left(\nabla V(x', 0)\right)$ for $x=(x', x_n)\in B_1$.

Equations \eqref{scaledvto1} and \eqref{EquV} now yield
$$\int_{\Om_1(0)} [\overline{\mathcal{A}}_{B_{1}}(\nabla V)- \overline{\mathcal{A}}_{B_{1}}(\nabla v)]\cdot\nabla\varphi dx -\int_{\Om_1(0)} f \frac{\partial \varphi}{\partial x_n} dx=0$$
for all $\varphi\in W^{1,\,p}_{0}(\Om_1(0))$.
Letting $\varphi=\phi^p (V-v)$, where $\phi$ is a standard cutoff function satisfying
$$\phi\in C^{\infty}_{0}(B_{1/4}), \quad 0\leq \phi\leq 1, \quad {\rm and}~ \phi \equiv 1 ~ {\rm on}~ \overline{B_{1/8}},$$
in the above identity we get
\begin{eqnarray*}
0&=&\int_{\Om_1(0)} \phi^p [\overline{\mathcal{A}}_{B_{1}}(\nabla V)- \overline{\mathcal{A}}_{B_{1}}(\nabla v)]\cdot(\nabla V-\nabla v) dx\\
&&+p\int_{\Om_1(0)} \phi^{p-1}(V-v) [\overline{\mathcal{A}}_{B_{1}}(\nabla V)- \overline{\mathcal{A}}_{B_{1}}(\nabla v)]\cdot \nabla \phi dx\\ 
&&-\int_{\Om_1(0)} \left[f \phi^p \frac{\partial (V-v)}{\partial x_n}  +p  f \phi^{p-1}(V-v) \frac{\partial \phi}{\partial x_n} \right]dx\\
&=:& J_1+J_2+J_3\end{eqnarray*}

By \eqref{sublinear} we have 
$$|J_2|\leq C \int_{\Om_1(0)} |V-v| (|\phi \nabla V|^{p-1}+ |\phi \nabla v|^{p-1}) |\nabla \phi| dx,$$
which by Young's inequality gives
$$|J_2|\leq \tau_2 \int_{\Om_1(0)} (|\phi \nabla V|^{p}+ |\phi \nabla v|^{p})dx + C(\tau_2) \int_{\Om_1(0)} |V-v|^p|\nabla \phi|^p dx
$$
for any $\tau_2>0$. Likewise,
\begin{eqnarray*}
|J_3|&\leq& \tau_3 \int_{\Om_1(0)} \left[|\phi \nabla (V-v)|^{p} + |\nabla\phi|^p |V- v|^{p}\right]dx\\
&& +\, C(\tau_3) \int_{\Om_1(0)}  |f|^\frac{p}{p-1} \phi^pdx
\end{eqnarray*}
for any $\tau_3>0$.

We now estimate the term $J_1$ from below. To do so we consider separately the case $p\geq 2$ and the case $p<2$. For $p\geq  2$ we immediately 
get from \eqref{monotone} that 
$$J_1 \geq c(\alpha, p) \int_{\Om_1(0)} |\phi \nabla(V-v)|^p dx.$$

For $p<2$ we write 
$$|\nabla (V-v)|^{p}= (|\nabla V|^2+ |\nabla v|^2)^{\frac{(p-2)p}{4}} |\nabla V- \nabla v|^p (|\nabla V|^2+ |\nabla v|^2)^{\frac{(2-p)p}{4}}$$ 
and use Young's inequality with exponents $\frac{2}{p}$ and $\frac{2}{2-p}>1$ to get
\begin{eqnarray*}
|\nabla (V-v)|^{p}&\leq& \tau_1 (|\nabla V|^p + |\nabla v|^p) +\\
&&+\, C(p)\tau_1^{\frac{p-2}{p}} (|\nabla V|^2+ |\nabla v|^2)^{\frac{(p-2)}{2}} |\nabla V- \nabla v|^2
\end{eqnarray*}
for any $\tau_1>0$.
Thus in view of \eqref{monotone} we get, for $p<2$,
\begin{eqnarray*}
J_1 &\geq& c(\alpha,p) \tau_1^{\frac{2-p}{p}} \int_{\Om_1(0)} |\phi \nabla(V-v)|^p dx\\
&& - \, c(\alpha,p)\tau_1^{\frac{2}{p}} \int_{\Om_1(0)} \phi^p (|\nabla V|^p + |\nabla v|^p) dx.
\end{eqnarray*}

Combining the estimates of $J_i$, $i=1, 2, 3$, and the equation $J_1=-J_2-J_3$ we arrive at
\begin{eqnarray*}
\lefteqn{c(\alpha, p) \int_{\Om_1(0)} |\phi \nabla(V-v)|^p dx}\\
&\leq& \tau_2 \int_{\Om_1(0)} (|\phi \nabla V|^{p}+ |\phi \nabla v|^{p})dx + C(\tau_2) \int_{\Om_1(0)} |V-v|^p|\nabla \phi|^p dx+\\
&&+\, \tau_3 \int_{\Om_1(0)} \left[|\phi \nabla (V-v)|^{p} + |\nabla\phi|^p |V- v|^{p}\right]dx +\\
&&+\, C(\tau_3) \int_{\Om_1(0)}  |f|^\frac{p}{p-1} \phi^pdx
\end{eqnarray*}
in the case $p\geq 2$, and 
\begin{eqnarray*}
\lefteqn{c(\alpha, p) \int_{\Om_1(0)} |\phi \nabla(V-v)|^p dx }\\
&\leq& \left(\tau_1^{\frac{p-2}{p}}\tau_2+c(\alpha,p)\tau_1\right) \int_{\Om_1(0)} (|\phi \nabla V|^{p}+ |\phi \nabla v|^{p})dx +\\
&& +\, C(\tau_1, \tau_2) \int_{\Om_1(0)} |V-v|^p|\nabla \phi|^p dx +\\
&& +\, \tau_1^{\frac{p-2}{p}}\tau_3 \int_{\Om_1(0)} \left[|\phi \nabla (V-v)|^{p} + |\nabla\phi|^p |V- v|^{p}\right]dx +\\
&& +\, C(\tau_1, \tau_3) \int_{\Om_1(0)}  |f|^\frac{p}{p-1} \phi^pdx
\end{eqnarray*}
in the case $p<2$. 

Choosing $\tau_2=\tau_3=\tau_1$ when $p\geq 2$ and $\tau_2=\tau_3=\tau_1^{\frac{2}{p}}$ when $p<2$, and then restricting $\tau_1<c(\alpha,p)/2$  we obtain
\begin{eqnarray*}
\lefteqn{\frac{1}{2}c(\alpha, p) \int_{\Om_1(0)} |\phi \nabla(V-v)|^p dx} \\
&\leq& C \tau_1 \int_{\Om_1(0)} (|\phi \nabla V|^{p}+ |\phi \nabla v|^{p})dx +\\
&& +\, C(\tau_1) \int_{\Om_1(0)} |V-v|^p|\nabla \phi|^p dx  +\, C(\tau_1) \int_{\Om_1(0)}  |f|^\frac{p}{p-1} \phi^pdx,
\end{eqnarray*}
for any $2-1/n <p\leq n$. Thus by our choice of $\phi$ and the bounds \eqref{uniboundw}-\eqref{LinftyboundforV} we obtain \begin{eqnarray*}
\lefteqn{c\int_{\Om_{1/8}(0)} |\nabla(V-v)|^p dx}\\
&\leq& C \tau_1 \int_{\Om_{1/4}(0)} (|\nabla V|^{p}+ |\nabla v|^{p})dx + \\
&& +\, C(\tau_1) \int_{\Om_{1/4}(0)}\left[ |V-v|^p +|f|^\frac{p}{p-1} \right]dx \\
&\leq& C\tau_1 +C(\tau_1)\left[ \eta^p+ \int_{\{x\in B_{1/4}: -4\delta<x_n<0\}} |\nabla V(x',0)|^p dx\right]\\
&\leq& C\tau_1 +C(\tau_1)\left(\eta^p+ C \delta\right).
\end{eqnarray*}

Finally, for any $\epsilon>0$ by choosing $\tau_1, \eta$, and then $\delta$  appropriately we get
$$\fint_{B_{1/8}} |\nabla(V-v)|^p dx=\frac{1}{|B_{1/8}|}\int_{\Om_{1/8}(0)} |\nabla(V-v)|^p dx \leq \epsilon^p,$$
which completes the proof of the lemma. 
\end{proof}

\begin{theorem}\label{compareVLinfty}
Suppose that $\mathcal{A}: \RR^n\times\RR^n\rightarrow \RR^n$ satisfies \eqref{sublinear}-\eqref{ellipticity}.
 For any $\epsilon > 0$ there exists a small $\delta = \delta (n, p, \alpha, \beta, \epsilon) > 0$ such that if $v \in W^{1,\, p}(\Omega_{\rho}(0))$ is a  solution of 
\begin{equation*}
\left\{ \begin{array}{rcl}
 {\rm div}~ \overline{\mathcal{A}}_{B_{\rho}(0)}(\nabla v)&=&0   \quad \text{in} ~\Om_{\rho}(0), \\ 
v&=& 0  \quad \text{on}~ \partial\Om\cap B_{\rho}(0),
\end{array}\right.
\end{equation*}
 with 
\begin{equation}\label{BrhoOmrho}
 B_{\rho}^{+}(0)\subset \Omega_{\rho}(0)\subset B_{\rho}(0)\cap \{ x_{n} > -4\delta\rho \},
\end{equation}
and the zero extension of $v$ from $\Omega_{\rho}(0)$ to $B_{\rho}(0)$ belongs to $W^{1,\, p}(B_{\rho}(0))$,
then there exists a weak solution $V\in W^{1,\, p}(B^{+}_{\rho}(0))$ of (\ref{perturbedpdeboundary}) whose  zero extension to $B_{\rho}(0)$   satisfies 
$$ \norm{\nabla V}^{p}_{L^{\infty}(B_{\rho/4}(0))} \leq C \fint_{B_{\rho}(0)}|\nabla v|^{p}dx$$
and
$$\fint_{B_{\rho/8}(0)}|\nabla v - \nabla V|^{p}dx \leq \epsilon^p \fint_{B_{\rho}(0)}|\nabla v|^{p}dx.$$
Here $C=C(n,p, \alpha, \beta)$.
\end{theorem}

\begin{proof}
Let  $\mathcal{B}(z, \eta)= \mathcal{A}(\rho z, \kappa \eta)/\kappa^{p-1},$
where
$\kappa^p=\fint_{B_{\rho}(0)}|\nabla v|^{p}dx.$
Then $\mathcal{B}$ satisfies conditions \eqref{sublinear}-\eqref{ellipticity} with the same constants $\alpha$ and $\beta$.
Moreover, it is easy to see that the function $h(z)=v(\rho z)/(\rho\kappa)$ satisfies the equation
\begin{equation*}
\left\{ \begin{array}{rcl}
 {\rm div}~ \overline{\mathcal{B}}_{B_{1}(0)}(\nabla h)&=&0   \quad \text{in} ~\Om^{\rho}_{1}(0), \\ 
h&=& 0  \quad \text{on}~ \partial\Om^{\rho}\cap B_{1}(0),
\end{array}\right.
\end{equation*}
where $\Om^{\rho}=\{z=x/\rho: x\in\Om\}$. By the choice of $\kappa$ we  have 
$$\left(\fint_{B_{1}(0)}|\nabla h|^{p}dz \right)^{1/p} = 1.$$

Also, the inclusions in \eqref{BrhoOmrho} now become 
\begin{equation*}
 B_{1}^{+}(0)\subset \Omega^{\rho}_{1}(0)\subset B_{1}(0)\cap \{ x_{n} > -4\delta \}.
\end{equation*}

Therefore, applying Lemma \ref{corB6+} for any $\epsilon>0$ there exist a constant $\delta=\delta(n, p, \alpha, \beta, \epsilon)>0$ and a function $G$ such that 
$$\fint_{B_{1/8}(0)}|\nabla h - \nabla G|^{p}dz \leq \epsilon^p, \qquad {\rm and~} \norm{\nabla G}^{p}_{L^{\infty}(B_{1/4}(0))} \leq C.$$

Now scaling back by choosing $V(x)= \kappa \rho G(x/\rho)$  we obtain the theorem. 
\end{proof}

\begin{corollary} \label{mainlocalestimate-boundary} Suppose that $\mathcal{A}: \RR^n\times\RR^n\rightarrow \RR^n$ satisfies \eqref{sublinear}-\eqref{ellipticity}. Let $s=\frac{\theta p}{(\theta-1)(p-1)}$ for some $\theta\in(1,\theta_0]$, where $\theta_0>1$ is as found in Lemma \ref{RVHbdry}. For any $\epsilon>0$ there exists a small $\delta=\delta(n, p, \alpha, \beta, \epsilon)>0$ such that the following holds. 
If $\Om$ is  $(\delta, R_0)$-Reifenberg flat and $u\in W_{0}^{1,\, p}(\Om)$ is a weak solution of 
\eqref{basicpde}  with $x_0\in\partial\Om$ and $0<R<R_0/10$ 
then there is a function $V\in  W^{1,\, \infty}(B_{R/10}(x_0))$ such that 
\begin{equation}\label{Linftyboundbdry}
\norm{\nabla V}_{L^{\infty}(B_{R/10}(x_0))}\leq C \fint_{B_{10R}(x_0)}|\nabla u| dx + C F(p, \nabla u,\mu, B_{10R}(x_0)),
\end{equation}
and 
\begin{eqnarray}\label{u-VV}
\lefteqn{\fint_{B_{R/10}(x_0)} |\nabla u-\nabla V|dx}\\
&\leq&   C \left(\epsilon + ([\mathcal{A}]^{R_0}_{s})^{\min\{1, \frac{1}{p-1}\}} \right)  \fint_{B_{10R}(x_0)}|\nabla u|dx +\nonumber\\
&& +\,  C \left(\epsilon + 1+ ([\mathcal{A}]^{R_0}_{s})^{\min\{1, \frac{1}{p-1}\}} \right)  F(p, 
\nabla u,\mu, B_{10R}(x_0)). \nonumber
\end{eqnarray}
Here $C=C(n, p, \alpha, \beta)$.
\end{corollary}

\begin{proof} 
Let  $x_0\in\partial\Om$, $0<R<R_0/10$  and  set
$\rho= R(1-\delta)$. By the remark made after Lemma \ref{DMbdry}, after a translation and a rotation, we may assume that $0\in\Om$, 
$x_0=(0, \dots, 0, -\delta \rho/(1-\delta))$  and the geometric setting

\begin{equation}\label{BrhoOmrhosecond}
 B_{\rho}^{+}(0)\subset \Omega_{\rho}(0)\subset B_{\rho}(0)\cap \{ x_{n} > -4 \delta\rho \}.
\end{equation}

Moreover, we have 
$$B_\rho(0)\subset B_{6\rho}(0)\subset B_{10R}(x_0),\quad {\rm and~} B_{R/10}(x_0)\subset B_{\rho/8}(0)$$
provided $\delta<1/45$.

Next we choose $w$ and $v$ as in \eqref{wapprox} and \eqref{vapprox} corresponding to these $R$ and $\rho$.
Then there holds
\begin{equation}\label{wtov}
\fint_{B_{\rho}(0)} |\nabla v|^p dx \leq C  \fint_{B_{\rho}(0)} |\nabla w|^p dx.
\end{equation}

For $V$ to be determined, we  observe that 
$$\fint_{B_{\rho}(0)} | \nabla u- \nabla V|dx=\fint_{B_{\rho}(0)} | \nabla (u-w) +\nabla (w-v)+\nabla (v-V)|dx.$$

 By Lemma \ref{DMbdry}
we find
\begin{eqnarray*}
\fint_{B_{\rho}(0)} | \nabla (u-w)|dx &\leq& C \fint_{B_{10R}(x_0)} | \nabla (u-w)|dx\\
&\leq& C F(p, \nabla u,\mu, B_{10R}(x_0)).
\end{eqnarray*}

Also by  Lemma  \ref{BMOapprox2} we have 
\begin{eqnarray*}
\fint_{B_{\rho}(0)} |\nabla (w-v)|dx\leq C ([\mathcal{A}]^{R_0}_{s})^{\min\{1, \frac{1}{p-1}\}} \fint_{B_{6\rho}(0)}|\nabla w| dx.
\end{eqnarray*}

On the other hand, by Theorem \ref{compareVLinfty} for any $\epsilon>0$ we can find a small positive $\delta=\delta(n, p, \alpha, \beta, \epsilon)<1/45$ such that,
under \eqref{BrhoOmrhosecond}, there is a  
function  $V\in W^{1,\, p}(B_\rho(0))\cap W^{1,\, \infty}(B_{\rho/2}(0))$ such that
$$
\norm{\nabla V}^{p}_{L^{\infty}(B_{\rho/4}(0))} \leq C \fint_{B_{\rho}(0)} |\nabla v|^p dx \leq  C  \fint_{B_{\rho}(0)} |\nabla w|^p dx,
$$
and
$$\fint_{B_{\rho/8}(0)} | \nabla (v- V)|^pdx \leq \epsilon^p \fint_{B_{\rho}(0)} |\nabla v|^p dx \leq  C  \epsilon^p\fint_{B_{\rho}(0)} |\nabla w|^p dx.$$
Here we  used \eqref{wtov} in the above inequalities.

At this point note that by Lemmas \ref{RVHbdry} and \ref{DMbdry} there holds
\begin{eqnarray*}
\left(\fint_{B_{\rho}(0)}|\nabla w|^pdx\right)^{1/p}&\leq& C \fint_{B_{6\rho}(0)}|\nabla w|dx\\
&\leq& C \fint_{B_{10R}(x_0)}|\nabla w|dx\\
&\leq& C\fint_{B_{10R}(x_0)}|\nabla u|dx + C F(p, 
\nabla u,\mu, B_{10R}(x_0)).
\end{eqnarray*}

It is then easy to see that these inequalities  yield \eqref{Linftyboundbdry} and \eqref{u-VV}.
\end{proof}

\section{Weighted estimates of upper-level sets}

We begin this section with the following technical result.

\begin{proposition}\label{Byun-Wang-int-bdry}
With $\mathcal{A}$ satisfying \eqref{sublinear}-\eqref{ellipticity}, there exist constants $\Lambda, s>1$, depending only on $n, p, \alpha,$ and $\beta$, such that the following holds. 
For any $\epsilon>0$ there exists a small $\delta=\delta(n, p, \alpha, \beta, \epsilon)>0$ such that if $u\in W_{0}^{1,\, p}(\Om)$ is a weak solution of 
\eqref{basicpde} where $\Om$ is $(\delta, R_0)$-Reifenberg flat,  $[\mathcal{A}]_s^{R_0}\leq \delta$,  and  
\begin{equation}\label{BWassum2}
B_{\rho}(y) \cap \{x\in \RR^n: {\rm\bf M}(|\nabla u|)\leq 1\}\cap \{ x\in \RR^n: {\rm\bf M}_{1}(\mu)^{\frac{1}{p-1}}\leq \delta \}\neq \emptyset
\end{equation}
for some ball $B_{\rho}(y)$ with $\rho< R_0/600$, then there holds
\begin{equation}\label{lambdadecay}
|\{ x\in \RR^n: {\rm\bf M}(|\nabla u|) > \Lambda\} \cap B_{\rho}(y)|< \epsilon \, |B_{\rho}(y)|.
\end{equation}
\end{proposition}

\begin{proof}  By \eqref{BWassum2} there exists $x_0\in B_\rho(y)$ such that for any $r>0$
\begin{equation}\label{MaxCon}
\fint_{B_r(x_0)} |\nabla u| dz \leq 1 \quad {\rm and}~ r\fint_{B_r(x_0)} d|\mu|  \leq \delta^{p-1}. 
\end{equation}
Here we recall that  both $u$ and $\mu$ are extended by zero outside $\Om$. 
We now claim that for $x\in B_\rho(y)$ there holds 
\begin{equation}\label{localizemax}
{\rm\bf M}(|\nabla u|)(x)  \leq \max\left\{{\rm\bf M}(\chi_{B_{2\rho}(y)}|\nabla u|)(x), 3^n  \right\}. 
\end{equation}
Indeed, for $r\leq \rho$ we have $B_r(x)\subset B_{2\rho}(y)$ and thus
$$\fint_{B_r(x)} |\nabla u| dz=\fint_{B_r(x)} \chi_{B_{2\rho}(y)} |\nabla u| dz, $$
whereas  for $r>\rho$ we have $B_r(x)\subset B_{3r}(x_0)$ and thus \eqref{MaxCon} yields 
$$\fint_{B_r(x)}  |\nabla u| dz \leq 3^n \fint_{B_{3r}(x_0)} |\nabla u| dz\leq 3^n.$$

We next observe that by  Remark \ref{notionofsol} regarding the notion of solutions we may assume that $u\in W^{1,\, p}_0(\Om)$.
A similar approximation procedure can be found in \cite[Remark 3.3]{Ph4}. 

In order to prove \eqref{lambdadecay} we consider separately the case  $B_{4\rho}(y)\Subset\Om$ and the case $\overline{B_{4\rho}(y)}\cap \partial\Om\not=\emptyset$.
First we consider the case that  $\overline{B_{4\rho}(y)}\cap\partial\Om\not=\emptyset$. Let $y_0\in \overline{B_{4\rho}(y)}\cap\partial\Om$.  We have 
\begin{equation}\label{yy0x0}
B_{2\rho}(y)\subset B_{6\rho}(y_0)\subset B_{600\rho}(y_0)\subset B_{605\rho}(x_0).
\end{equation}

Since $\rho<R_0/600$ by Corollary \ref{mainlocalestimate-boundary} for any $\eta\in (0,1)$ there is a small $\delta=\delta(\eta)>0$ 
such that if $\Om$ is $(\delta, R_0)$-Reifenberg flat then one can find
a constant $s=s(n, p, \alpha, \beta)>1$ and a function $V\in W^{1,\, p}(B_{6\rho}(y_0))\cap W^{1,\, \infty}(B_{6\rho}(y_0))$ with
$$\norm{\nabla V}_{L^{\infty}(B_{6\rho}(y_0))}\leq C \fint_{B_{600\rho}(y_0)}|\nabla u| dx + C F(p, \nabla u,\mu, B_{600\rho}(y_0))$$
and 
\begin{eqnarray*}
\lefteqn{\fint_{B_{6\rho}(y_0)} |\nabla u-\nabla V|dx}\\
&\leq&   C \left(\eta + ([\mathcal{A}]^{R_0}_{s})^{\min\{1, \frac{1}{p-1}\}} \right)  \fint_{B_{600\rho}(y_0)}|\nabla u|dx +\\
&& +\,  C \left(\eta + 1+ ([\mathcal{A}]^{R_0}_{s})^{\min\{1, \frac{1}{p-1}\}} \right)  F(p, 
\nabla u,\mu, B_{600\rho}(y_0)).
\end{eqnarray*}

Thus it follows from \eqref{MaxCon} and \eqref{yy0x0} that 
\begin{eqnarray}\label{VC_0bound}
\lefteqn{\norm{\nabla V}_{L^{\infty}(B_{2\rho}(y))}\leq \norm{\nabla V}_{L^{\infty}(B_{6\rho}(y_0))}}\\
&\leq& C \fint_{B_{605\rho}(x_0)}|\nabla u| dx + C F(p, \nabla u,\mu, B_{605\rho}(x_0))\nonumber\\
&\leq& C_0,\nonumber
\end{eqnarray}
and also
\begin{eqnarray}\label{uVetadelta}
\lefteqn{\fint_{B_{2\rho}(y)} |\nabla u-\nabla V|dx\leq C \fint_{B_{6\rho}(y_0)} |\nabla u-\nabla V|dx}\\
&\leq&   C \left(\eta + ([\mathcal{A}]^{R_0}_{s})^{\min\{1, \frac{1}{p-1}\}} \right)  \fint_{B_{605\rho}(x_0)}|\nabla u|dx +\nonumber\\
&& +\,  C \left(\eta+ 1+ ([\mathcal{A}]^{R_0}_{s})^{\min\{1, \frac{1}{p-1}\}} \right)  F(p, \nabla u,\mu, B_{605\rho}(x_0)) \nonumber\\
&\leq&   C \left(\eta + ([\mathcal{A}]^{R_0}_{s})^{\min\{1, \frac{1}{p-1}\}} \right)  +\nonumber\\
&&+\,  C \left(\eta+ 1+ ([\mathcal{A}]^{R_0}_{s})^{\min\{1, \frac{1}{p-1}\}} \right)(\delta +\delta^{p-1}). \nonumber
\end{eqnarray}
Here $C_0$ and $C$ depend only on $n, p, \alpha,$ and $\beta$. In view of \eqref{localizemax} and \eqref{VC_0bound} we see that 
for $\Lambda =\max\{3^n, 2C_0\}$ there holds
\begin{eqnarray*}
\lefteqn{|\{ x\in \RR^n: {\rm\bf M}(|\nabla u|) > \Lambda\} \cap B_{\rho}(y)|}\\
&\leq& |\{ x\in \RR^n: {\rm\bf M}(\chi_{B_{2\rho}(y)}|\nabla u|) > \Lambda\} \cap B_{\rho}(y)|\nonumber\\
&\leq & \left|\{ x\in \RR^n: {\rm\bf M}(\chi_{B_{2\rho}(y)}|\nabla V|) > \Lambda/2\} \cap B_{\rho}(y)\right|\nonumber\\
&& +\, \left|\{ x\in \RR^n: {\rm\bf M}(\chi_{B_{2\rho}(y)}|\nabla u- \nabla V|) > \Lambda/2\} \cap B_{\rho}(y)\right| \nonumber\\
&\leq&  \left|\{ x\in \RR^n: {\rm\bf M}(\chi_{B_{2\rho}(y)}|\nabla u- \nabla V|) > \Lambda/2\} \cap B_{\rho}(y)\right|. \nonumber
\end{eqnarray*}

Thus by weak-type $(1,1)$ bound for the Hardy-Littlewood maximal function and inequality \eqref{uVetadelta} we find
\begin{eqnarray*}
\lefteqn{|\{ x\in \RR^n: {\rm\bf M}(|\nabla u|) > \Lambda\} \cap B_{\rho}(y)|}\\
&\leq& \frac{C}{\Lambda} \int_{B_{2\rho}(y)} |\nabla u-\nabla V|dx\\
&\leq& \frac{C}{C_0} |B_{2\rho}(y)| \left[C \left(\eta + ([\mathcal{A}]^{R_0}_{s})^{\min\{1, \frac{1}{p-1}\}} \right)\right.\\
&& \left. +\,  C \left(\eta+ 1+ ([\mathcal{A}]^{R_0}_{s})^{\min\{1, \frac{1}{p-1}\}} \right)(\delta +\delta^{p-1}) \right]\\
&<& \epsilon \, |B_{\rho}(y)|
\end{eqnarray*}
for any given $\epsilon>0$, provided $[\mathcal{A}]^{R_0}_{s}\leq \delta$ and $\eta, \delta$ are appropriately chosen.

This gives the estimate \eqref{lambdadecay} in the case $\overline{B_{4\rho}(y)}\cap \partial\Om\not=\emptyset$.
The case   $B_{4\rho}(y)\Subset\Om$ can be done in a similar way using Corollary \ref{mainlocalestimate-interior} instead 
of Corollary \ref{mainlocalestimate-boundary}.
\end{proof}


We next obtain a weighted version of Proposition \ref{Byun-Wang-int-bdry}. 

\begin{proposition}\label{weightedByun-Wang}
Let $w$ be an $A_\infty$ weight in $\RR^n$ and let $\mathcal{A}$ satisfy \eqref{sublinear}-\eqref{ellipticity}. There exist constants $\Lambda, s>1$,  depending only on $n, p, \alpha,$ and $\beta$, such that the following holds. 
For any $\epsilon>0$ there exists a small $\delta=\delta(n, p,\alpha, \beta, \epsilon, [w]_{A_\infty})>0$ such that if $u\in W_{0}^{1,\, p}(\Om)$ is a weak solution of 
\eqref{basicpde} where $\Om$ is $(\delta, R_0)$-Reifenberg flat,  $[\mathcal{A}]_s^{R_0}\leq \delta$,  and  
\begin{equation*}
B_{\rho}(y) \cap \{x\in \RR^n: {\rm\bf M}(|\nabla u|)\leq 1\}\cap \{ x\in \RR^n: {\rm\bf M}_{1}(\mu)^{\frac{1}{p-1}}\leq \delta \}\neq \emptyset
\end{equation*}
for some ball $B_{\rho}(y)$ with $\rho< R_0/600$, then there holds
\begin{equation*}
w(\{ x\in \RR^n: {\rm\bf M}(|\nabla u|) > \Lambda\} \cap B_{\rho}(y))< \epsilon \, w(B_{\rho}(y)).
\end{equation*}
\end{proposition}

\begin{proof}
Suppose that $(A, \nu)$ is a pair of $A_\infty$ constants of $w$.  Let $\epsilon> 0$ be given and choose $\delta =\delta(\epsilon, A, \nu)$ as in Proposition \ref{Byun-Wang-int-bdry} with $(\epsilon/(2A))^{1/\nu}$ replacing $\epsilon$. Then there exist  $\Lambda, s>1$,  
depending only on $n, p, \alpha,$ and $\beta$, such that if $\Om$ is $(\delta, R_0)$-Reifenberg flat and  $\mathcal{A}$ satisfies  the
$(\delta, R_0)$-BMO condition with exponent $s$, then there holds 
\begin{equation}\label{lebesgue}
|\{ x\in \RR^n: {\rm\bf M}(|\nabla u|)> \Lambda\} \cap B_{\rho}(y)| \leq \left(\frac{\epsilon}{2A}\right) ^{1/\nu}|B_{\rho}(y)|.
\end{equation}

Thus,  using the $A_{\infty}$ characterization of $w$, we get from \eqref{lebesgue} that
\begin{eqnarray*}
\lefteqn{ w(\{ x\in \RR^n: {\rm\bf M}(|\nabla u|) > \Lambda\} \cap B_{\rho}(y))}\\
&\leq& A \left[ \frac{|\{ x\in \RR^n: {\rm\bf M}(|\nabla u|) > \Lambda\} \cap B_{\rho}(y)| }{|B_{\rho}(y)|}\right ]^{\nu} w(B_{\rho}(y))\\
&\leq&  \frac{\epsilon}{2} \, w(B_{\rho}(y)) < \epsilon \, w(B_{\rho}(y)).
\end{eqnarray*}

This completes the proof of the proposition.
\end{proof}

The following result is just the contrapositive of Proposition  \ref{weightedByun-Wang}.
\begin{proposition}\label{contra}
Let $w$ be an $A_\infty$ weight in $\RR^n$ and let $\mathcal{A}$ satisfy \eqref{sublinear}-\eqref{ellipticity}. There exist constants $\Lambda, s>1$,  depending only on $n, p, \alpha,$ and $\beta$, such that the following holds. 
For any $\epsilon>0$ there exists a small $\delta=\delta(n, p, \alpha, \beta, \epsilon, [w]_{A_\infty})>0$ such that if $u\in W_{0}^{1,\, p}(\Om)$ is a weak solution of 
\eqref{basicpde} where $\Om$ is $(\delta, R_0)$-Reifenberg flat,  $[\mathcal{A}]_s^{R_0}\leq \delta$,  and  
\begin{equation*}
w(\{ x\in \RR^n: {\rm\bf M}(|\nabla u|) > \Lambda\} \cap B_{\rho}(y))\geq \epsilon \, w(B_{\rho}(y))
\end{equation*}
for some ball $B_{\rho}(y)$ with $\rho< R_0/600$, then there holds
\begin{equation*}
B_{\rho}(y) \subset \{x\in \RR^n: {\rm\bf M}(|\nabla u|)> 1\}\cup \{ x\in \RR^n: {\rm\bf M}_{1}(\mu)^{\frac{1}{p-1}}> \delta \}.
\end{equation*}
\end{proposition}

Proposition \ref{contra} is designed so that the following technical lemma can be applied. 
A proof of this lemma, which  uses Lebesgue Differentiation Theorem and the
standard Vitali covering lemma, can be found in \cite{M-P1}. See also \cite{W, BW1}. In a sense this lemma plays a role similar to that of the 
famous  Calder\'on-Zygmund-Krylov-Safonov decomposition.

\begin{lemma}\label{vitalibdry}
 Let $\Omega $ be a $(\delta, R_0)$-Reifenberg flat domain with $\delta<1/8$, and let $w$ be an $A_{\infty}$
 weight.
 Suppose that $\{ B_{r}(y_i)\}_{i=1}^{L}$ is a sequence of balls with centers $y_i\in \Om$
 and  radius $r<R_0/4$  that covers $\Om$.  Let  $C\subset D\subset \Omega$
 be measurable sets for which there exists $0<\epsilon < 1 $ such that
\begin{enumerate}
\item $ w(C)< \epsilon\, w(B_{r}(y_i)) $ for all $i=1,\dots, L$, and
\item for all $y \in \Omega$ and $\rho \in (0, 2r]$, if  $w(C\cap B_{\rho}(y))\geq \epsilon\,
w(B_{\rho}(y))$, then $B_{\rho}(y)\cap \Omega \subset D.$
\end{enumerate}

Then there holds 
 \[ w(C) \leq B\, \epsilon \, w(D)\]
for a constant $B$ depending only on $n$ and $[w]_{A_\infty}$.
\end{lemma}

Proposition \ref{contra} and Lemma \ref{vitalibdry} yield the following result.
\begin{theorem}\label{technicallemma}
With $\mathcal{A}$ satisfying \eqref{sublinear}-\eqref{ellipticity}, let $w$ be an $A_\infty$ weight and let 
$s$, $\Lambda >1$ be as in  Proposition \ref{contra}.
Then for any $\epsilon>0$ there exists $\delta=\delta(n, p, \alpha, \beta, \epsilon, [w]_{A_\infty})>0$ such that the following holds.
Suppose that $u \in W^{1, \, p}_{0}(\Omega)$ is a weak solution of (\ref{basicpde}) 
in a $(\delta, R_0)$-Reifenberg flat domain $\Om$
and  $[\mathcal{A}]_s^{R_0}\leq \delta$.
Suppose also that  $\{ B_{r}(y_i)\}_{i=1}^{L}$ is a sequence of balls with centers $y_i\in \Om$ and a common  radius $0<r< R_0/1200$ that covers $\Omega$. If
for all $i=1,\dots, L$
\begin{equation}\label{hypo1bdry}
w(\{x\in \Om: {\rm\bf M}(|\nabla u|) > \Lambda\}) < \epsilon \, w(B_{r}(y_{i})),
\end{equation}
then for any $t>0$ and any integer $k\geq 1$ there holds
\begin{eqnarray*}
\lefteqn{w(\{x\in \Omega: {\rm\bf M}(|\nabla u|) > \Lambda ^{k}\})^{t}}\\
&\leq& \sum_{i=1}^{k} (B \, \epsilon^t)^{i} \, w(\{ x\in \Omega : {\rm\bf M}_1(\mu)^{\frac{1}{p-1}} > \delta \Lambda^{k-i}\})^{t}\\
&& + \, (B \, \epsilon^t)^{k} \, w(\{x\in \Omega : {\rm\bf M}(|\nabla u|) > 1\})^{t},
\end{eqnarray*}
where the constant $B=B(n, t, [w]_{A_\infty})$.
\end{theorem}

\begin{proof} 
The theorem will be proved by induction in $k$. Given $\epsilon > 0$, we take $\delta=\delta(\epsilon, [w]_{A_\infty})$ as in Proposition \ref{contra}. The case $k=1$ follows from 
Proposition \ref{contra} and Lemma  \ref{vitalibdry}. Indeed, let 
$$C = \{ x\in \Omega: {\rm\bf M}(|\nabla u|) > \Lambda\}$$
 and
$$D =  \{x\in \Omega: {\rm\bf M}(|\nabla u|) > 1 \}\cup \{ x\in \Omega: [{\rm\bf M}_{1}(\mu)]^{\frac{1}{p-1}} > \delta \}.$$

Then  from assumption (\ref{hypo1bdry}) it follows that
$w(C) < \epsilon \, w(B_{r}(y_{i}))$
for all $i=1,\dots, L$.  Moreover, if $y\in \Omega$ and $\rho \in (0, 2r)$ such that $w(C\cap B_{\rho}(y) ) \geq \epsilon\, w(B_{\rho}(y)), $
then $0< \rho < R_0/600$ and  $B_{\rho}(y) \cap \Omega\subset D$ by Proposition \ref{contra}. Thus  the hypotheses of Lemma \ref{vitalibdry} 
are satisfied which yield
\begin{eqnarray*}
w(C)^{t}&\leq& B(n, t, [w]_\infty)\, \epsilon^{t}\, w(D)^{t}\\
&\leq& B(n, t, [w]_\infty)\, \epsilon^{t}\, w(\{x\in \Omega: {\rm\bf M}(|\nabla u|) > 1 \})^t +\\
&& +\, B(n, t, [w]_\infty)\, \epsilon^{t}\,  w(\{ x\in \Omega: {\rm\bf M}_1(\mu)^{\frac{1}{p-1}} > \delta\})^t. 
\end{eqnarray*}

This proves the case $k=1$. Suppose now that the conclusion of the lemma holds for some $k>1$. Normalizing $u$ to
$u_{\Lambda}  = u/\Lambda$ and $\mu_{\Lambda} = \mu/\Lambda^{p-1}$, we see that for every $i=1, \dots, L$ there holds 
\[
\begin{split}
w(\{ x\in \Omega:  {\rm\bf M}(|\nabla u_{\Lambda}|) > \Lambda \} )
&= w(\{ x\in \Omega : {\rm\bf M}(|\nabla u|) > \Lambda^{2} \} )\\
&\leq w(\{ x\in \Omega : {\rm\bf M}(|\nabla u|) > \Lambda \} )\\
 &< \epsilon \, w(B_{r}(y_{i})).
\end{split}
\]
Here we  used   $\Lambda>1$ in the first inequality. By inductive hypothesis it follows that
\begin{eqnarray*}
\lefteqn{w(\{ x\in \Omega : {\rm\bf M}(|\nabla u|) > \Lambda^{k+1}\})^t}\\
&=& w(\{ x\in \Omega : {\rm\bf M}(|\nabla u_{\Lambda}|) > \Lambda^{k} \})^t \nonumber  \\
&\leq& \sum_{i=1}^{k} (B\, \epsilon^t)^{i} \, w(\{ x\in \Omega: {\rm\bf M}_1(\mu_{\Lambda})^{\frac{1}{p-1}} > \delta \Lambda^{k-i}\})^t \nonumber\\
&& + \, (B\, \epsilon^t)^{k} \, w(\{ x\in \Omega: {\rm\bf M}(|\nabla u_{\Lambda}|) > 1 \})^t \nonumber\\
&=& \sum_{i=1}^{k} (B\, \epsilon^t)^{i} \, w(\{ x\in \Omega: {\rm\bf M}_1(\mu)^{\frac{1}{p-1}} > \delta\Lambda^{k+1-i}\})^t \nonumber \\
&& + \, (B\, \epsilon^t)^{k} \, w(\{ x\in \Omega: {\rm\bf M}(|\nabla u|) > \Lambda \})^t. \nonumber
\end{eqnarray*}

Now  applying the case $k = 1$ to the last term  we conclude that
\begin{eqnarray*}
\lefteqn{w(\{ x\in \Omega : {\rm\bf M}(|\nabla u|) > \Lambda^{k+1} \})^t}\\
&\leq& \sum_{i=1}^{k+1} (B\, \epsilon^t)^{i} \, w(\{ x\in \Omega: {\rm\bf M}_{1}(\mu)^{\frac{1}{p-1}} > \delta \Lambda^{k+1-i}\})^t\\
&& + \, (B\, \epsilon^t)^{k+1 } \, w(\{ x\in \Omega: {\rm\bf M}(|\nabla u|) > 1 \})^t.
\end{eqnarray*}

This completes the proof of the theorem.
\end{proof}

\section{Global Muckenhoupt-Wheeden type  estimates}\label{sec5}

We are now ready to prove the first main theorem of the paper.
\subsection*{Proof of Theorem \ref{main}}  We will consider only the case $t\not=\infty$ as for
$t=\infty$ the proof is just similar.  We begin by selecting
a finite collection of points $\{y_{i}\}_{i=1}^{L}\subset \Omega$
and a ball $B_0$ of radius $2 {\rm diam}(\Om)$ such that
\[
\Omega \subset \bigcup_{i= 1}^{L} B_{r}(y_{i})\subset
B_0,
\]
where $r= \min\{R_0/1200, {\rm diam}(\Om)\}$. We claim that we can choose $N$
large  such that
for $u_{N} = u/N$ and for all $ i = 1,\dots, L$
\begin{equation}\label{smallness}
w(\{ x\in \Omega: {\rm\bf M}(|\nabla u_{N}|)(x) > \Lambda
\}) < \epsilon \, w(B_{r}(y_{i})).
\end{equation}
Here $\Lambda=\Lambda(n, p, \alpha, \beta)$ is as in Proposition \ref{contra}. To show this we may assume that  $\|\nabla u\|_{L^{1}(\Omega)} >0$. Then by weak type $(1, 1)$ estimate for the  maximal function  there exists a constant $C_{{\rm\bf M}}=C_{{\rm\bf M}}(n)>0$  such that  
\begin{equation*}
|\{x \in \Omega: {\rm\bf M}(|\nabla u_{N}|)(x) > \Lambda\}|
< \frac{C_{{\rm\bf M}}}{\Lambda N}\int_{\Omega}|\nabla u|dx. 
\end{equation*}

By Definition \ref{inversedoubling} this yields
\begin{eqnarray*}
\lefteqn{w(\{ x\in \Omega: {\rm\bf M}(|\nabla u_{N}|)(x) > \Lambda \})}\\
&<& A\left[\frac{C_{{\rm\bf M}}}{\Lambda N |B_0|}\int_{\Omega}|\nabla u|dx\right]^{\nu} w(B_0)\\
&=&A\left[\frac{C_{{\rm\bf M}}}{\Lambda N |B_0|}\int_{\Omega}|\nabla u|dx\right]^{\nu} \frac{w(B_0)}{w(B_r(y_i))}w(B_r(y_i)),
\end{eqnarray*}
where  $(A, \nu)$ is a pair of $A_{\infty}$ constants of $w$. Also, as $w\in A_{\infty}$, it is known that (see, e.g., \cite{Gra}) there exist $C_1=C_1(n, A, \nu)$ and 
$p_1=p_1(n, A, \nu)>1$ such that 
\[ \frac{w(B_0)}{w(B_{r}(y_{i}))} \leq C_1 \left[\frac{|B_0|}{|B_r(y_i)|} \right]^{p_1}. \]

Thus one obtains \eqref{smallness}  provided we choose  $N>0$ so  that 
\begin{eqnarray}\label{Neqn}
 \frac{C_{{\rm\bf M}}}{\Lambda N |B_0|}\int_{\Omega}|\nabla u|dx  &=& \left(\frac{\epsilon}{C_1 A} \right)^{1/\nu} \left[\frac{|B_r(y_i)|}{|B_0|} \right]^{p_1/\nu}\\
 &=& \left(\frac{\epsilon}{C_1 A} \right)^{1/\nu} \left[r/2{\rm diam}(\Om) \right]^{n p_1/\nu}. \nonumber
\end{eqnarray}

Note that for  $0<t<\infty$ we have 
$$C^{-1} S\leq \norm{{\rm\bf M}(|\nabla u_{N}|)}^{t}_{L_{w}^{q,\, t}(\Om)}\leq C(w(\Om)^{\frac{t}{q}}+S),$$
where $S$ is the sum
\begin{equation*}
S= \sum_{k = 1}^{\infty}{\Lambda}^{t k}w(\{ x\in \Om: {\rm\bf M}(|\nabla u_{N}|)(x) > {\Lambda}^{k} \})^{\frac{t}{q}}.
\end{equation*}

By  \eqref{smallness} and  Theorem \ref{technicallemma} we find, with $\mu_N=\mu/N^{p-1}$, 
\begin{eqnarray*}
S &\leq & \sum_{k = 1}^{\infty} \Lambda^{tk} \left[ \sum_{j=1}^{k} (B\, \epsilon^{\frac{t}{q}})^{j}w(\{ x\in \Om: {\rm\bf M}_1(\mu_{N})^{\frac{1}{p-1}} > \delta \Lambda^{k-j}\})^{\frac{t}{q}} \right]\\
&& + \sum_{k = 1}^{\infty} \Lambda^{tk} (B\, \epsilon^{\frac{t}{q}})^{k}  w(\{x\in \Om: {\rm\bf M}(|\nabla u_{N}|)(x) > 1\})^{\frac{t}{q}}.
\end{eqnarray*}
Here the constants $B=B(n, t/q, [w]_{A_\infty})$, $\epsilon=\Lambda^{-q} B^{-q/t} 2^{-q/t}$, and $\delta=\delta(n, p, \alpha, \beta, \epsilon, [w]_{A_\infty})$ is determined by 
Theorem \ref{technicallemma} which ultimately depends only on $n, p, \alpha, \beta, t, q,$ and $[w]_{A_\infty}$.

Thus  we have 
\begin{eqnarray*}
S &\leq & \sum_{j=1}^{\infty} (\Lambda^{t}  B\, \epsilon^{\frac{t}{q}})^{j}  \left[ \sum_{k=j}^{\infty} \Lambda^{t(k-j)}w(\{ x\in \Om: {\rm\bf M}_1(\mu_{N})^{\frac{1}{p-1}} > \delta \Lambda^{k-j}\})^{\frac{t}{q}}\right]\\
&&+\sum_{k=1}^{\infty} (\Lambda^{t}  B\, \epsilon^{\frac{t}{q}})^{k}w(\{x\in \Om: {\rm\bf M}(|\nabla u_{N}|)(x) > 1\})^{\frac{t}{q}}\\
&\leq& C\left [\|{\rm\bf M}_1(\mu_{N})\|^{t}_{L^{q,\, t}_{w}(\Om)}+ w(\Om)^{\frac{t}{q}}\right]\sum_{k=1}^{\infty}(\Lambda^{t}  B\, \epsilon^{\frac{t}{q}})^{k}\\
&\leq& C \left [ \|{\rm\bf M}_1(\mu_{N})\|^{t}_{L^{q,\, t}_{w}(\Om)}+ w(\Om)^{\frac{t}{q}} \right],
\end{eqnarray*}
by our choice of $\epsilon$. Here $C$ depends only on  $n, p, \alpha, \beta, q,  t$, and   $[w]_{A_\infty}$.

Thus  we obtain 
\[
\|{\rm\bf M}(|\nabla u_{N}|)\|^{t}_{L^{q,\, t}_{w}(\Om)}\leq C\left [w(\Om)^{\frac{t}{q}} + \|{\rm\bf M}_1(\mu_{N})^{\frac{1}{p-1}}\|^{t}_{L^{q,\, t}_{w}(\Om)}\right].
\]

 This gives 
\begin{equation}\label{wNpp}
\|\nabla u\|_{L^{q,\, t}_{w}(\Om)}\leq C\left[ N w(\Om)^{\frac{1}{q}} + \|{\rm\bf M}_1(\mu)^{\frac{1}{p-1}}\|_{L^{q,\, t}_
{w}(\Om)}\right].
\end{equation}

On the other hand, by  \eqref{Neqn} and the condition $p> 2-\frac{1}{n}$ we get 
\begin{eqnarray*}
N &\leq& C\, |B_0|^{-1}  \|\nabla u\|_{L^{1}(\Omega)}\\
&\leq& C\, {\rm diam}(\Om)^{-n} |\Om|^{1-\frac{n-1}{n(p-1)}} |\mu|(\Omega)^{\frac{1}{p-1}}\\
&\leq& C\,  \left[\frac{|\mu|(\Omega)}{{\rm diam}(\Om)^{n-1}}\right]^{\frac{1}{p-1}},
\end{eqnarray*}
where the second inequality follows from standard estimates for equations with measure data (see, e.g., \cite{BBGGPV, DMOP}). Thus for any $x\in \Om$ we have 
$$N\leq C\, [{\rm\bf M}_1(\mu)(x)]^{\frac{1}{p-1}},$$
where $C$ depends only on  $n, p, \alpha, \beta, q, t, [w]_{A_\infty}$, and ${\rm diam}(\Om)/R_0$.

Finally, combining  the last inequality with  \eqref{wNpp}  we obtain \eqref{B0bound} as desired.

\section{Applications to quasilinear Riccati type equations}\label{app}

The main goal of this section is to prove  Theorem \ref{main-Ric}. First, we  recall the definition of renormalized solutions.
Let $\mathcal{M}_{B}(\Om)$ be the set of all signed measures in $\Om$ with bounded total variations. We  denote by $\mathcal{M}_{0}(\Om)$ (respectively $\mathcal{M}_{s}(\Om)$) the set of all measures in
 $\mathcal{M}_{B}(\Om)$
which are absolutely continuous (respectively singular) with
respect to the capacity ${\rm cap}_{p}(\cdot,\Om)$. Here ${\rm cap}_{p}(\cdot,\Om)$ is the
capacity relative to the domain $\Om$ defined by
\begin{equation*}
{\rm cap}_{p}(K,\Om)=\inf\Big\{\int_{\Om}\m{\nabla \phi}^{p}dx: \phi\in C_{0}^{\infty}
(\Om), \phi\geq 1 {\rm ~on~} K \Big\}
\end{equation*}
for any compact set $K\subset\Om$. Thus every measure in $\mathcal{M}_{s}(\Om)$ is supported on a Borel set $E$ with ${\rm cap}_{p}(E,\Om)=0$.
Recall from  \cite[Lemma 2.1]{FST} that, for every measure $\mu$ in $\mathcal{M}_{B}
(\Om)$, there exists a unique pair of measures ($\mu_{0}, \mu_{s}$) with $\mu_{0}\in \mathcal
{M}_{0}(\Om)$ and  $\mu_{s}\in \mathcal{M}_{s}(\Om)$, such that $\mu=\mu_{0}+\mu_{s}$.

 For a measure $\mu$ in $\mathcal{M}_{B}(\Om)$, its positive and
negative parts are denoted by $\mu^{+}$ and $\mu^{-}$, respectively.
We say that  a sequence of measures $\{\mu_{k}\}$ in
$\mathcal{M}_{B}(\Om)$ converges in the narrow topology to $\mu \in
\mathcal{M}_{B}(\Om)$ if
$$\lim_{k\rightarrow\infty}\int_{\Om}\varphi \, d\mu_{k}=\int_{\Om}\varphi \,
d\mu$$
for every bounded and continuous function $\varphi$ on $\Om$. 

The notion of renormalized solutions is a generalization of that of entropy solutions introduced in \cite{BBGGPV} and \cite{BGO}, where the measure data are 
assumed to be  in $L^1(\Om)$ or in  $\mathcal{M}_{0}(\Om)$.
Several equivalent definitions of renormalized solutions
were  given  in \cite{DMOP}, two of which are  the following ones.

\begin{definition}\label{rns1}{\rm
Let $\mu \in \mathcal{M}_{B}(\Om)$. Then $u$ is said to be a renormalized
solution of
\begin{eqnarray}\label{liDirichlet}
\left\{\begin{array}{rcl}
-{\rm div}\,\mathcal{A}(x, \nabla u)&=&\mu \quad {\rm in}~ \Om,\\
u&=&0\quad {\rm on}~ \partial\Om,
\end{array}
\right.
\end{eqnarray}
if the following conditions hold:
\begin{itemize}

\item[(a)] The function $u$ is measurable and finite almost everywhere, and $T_{k}(u)$ belongs
to $W_{0}^{1,\,p}(\Om)$ for every $k>0$.

\item[(b)] The gradient $\nabla u$ of $u$ satisfies $\m{\nabla u}^{p-1}\in L^{q}(\Om)$ for all $q<\frac{n}{n-1}$.

\item[(c)] If $w$ belongs to $W_{0}^{1,\,p}(\Om)\cap L^{\infty}(\Om)$ and if there exist 
$w^{+\infty}$ and $w^{-\infty}$ in $W^{1,\,r}(\Om)\cap L^{\infty}(\Om)$, with $r>n$, such that
\begin{equation*}
\left\{\begin{array}{c}
w=w^{+\infty} \hspace*{.2in}{\rm a.e. ~on~ the~ set}~~ \{u>k\},\\
w=w^{-\infty} \hspace*{.2in}{\rm a.e. ~on ~the ~set}~~ \{u<-k\}
\end{array}
\right.
\end{equation*}
for some $k>0$ then
\begin{equation*}
\int_{\Om}\mathcal{A}(x, \nabla u)\cdot\nabla wdx=\int_{\Om}w d\mu_{0}+\int_{\Om}w^{+\infty}
d\mu_{s}^{+}-\int_{\Om}w^{-\infty}d\mu_{s}^{-}.
\end{equation*}
\end{itemize}
}\end{definition}

\begin{definition}\label{rns}{\rm
Let $\mu \in \mathcal{M}_{B}(\Om)$. Then  $u$ is a renormalized
solution of (\ref{liDirichlet})
if $u$ satisfies (a) and (b) in Definition \ref{rns1}, and if the following conditions hold:
\begin{itemize}
\item[(c)] For every $k>0$ there exist two nonnegative measures in $\mathcal{M}_{0}(\Om)$,
$\lambda^{+}_{k}$ and
$\lambda^{-}_{k}$, concentrated on the sets $\{u=k\}$ and $\{u=-k\}$, respectively, such that
$\lambda^{+}_{k}\rightarrow \mu_{s}^{+}$ and
$\lambda^{-}_{k}\rightarrow \mu_{s}^{-}$ in the narrow topology
of measures.

\item[(d)] For every $k>0$
\begin{equation}\label{truncate}
\int_{\{\m{u}<k\}}\mathcal{A}(x, Du)\cdot\nabla\varphi dx= \int_{\{\m{u}<k\}}\varphi
d\mu_{0} +  \int_{\Om}\varphi d\lambda_{k}^{+} -  \int_{\Om}\varphi d\lambda_{k}^{-}
\end{equation}
for every $\varphi$ in ${\rm W}_{0}^{1,\,p}(\Om)\cap L^{\infty}(\Om)$.
\end{itemize}
}\end{definition}

\begin{remark}\label{quasi}{\rm
By  \cite[Remark 2.18]{DMOP}, if $u$ is a renormalized solution of (\ref{liDirichlet})
then (the ${\rm cap}_{p}$-quasi continuous representative
of) $u$ is finite quasieverywhere with respect to ${\rm cap}_{p}(\cdot, \Om)$.
Therefore, $u$ is finite $\mu_{0}$-almost everywhere.
}\end{remark}
\begin{remark}\label{approx}{\rm
By \eqref{truncate}, if $u$ is a renormalized solution  of (\ref{liDirichlet}) then
\begin{equation*}
-{\rm div}\mathcal{A}(x,\nabla T_{k}(u))=\mu_{k} \quad {\rm in~} \Om,
\end{equation*}
with
$$\mu_{k}=\chi_{\{\m{u}<k\}}\mu_{0}+\lambda^{+}_{k}-\lambda^{-}_{k}.$$

Since
$T_{k}(u)\in W_{0}^{1,\,p}(\Om)$, by \eqref{sublinear} we see that
$\mu_{k}$ belongs to
the dual space of $W_{0}^{1,\,p}(\Om)$. Moreover, by Remark \ref{quasi},
$\m{u}<\infty$ $\mu_{0}$-almost everywhere and hence $\chi_{\{\m{u}<k\}} \rightarrow
\chi_{\Om}$ $\mu_{0}$-almost everywhere as $k\rightarrow\infty$. Therefore, $\mu_{k}$ (resp. $|\mu_k|$) converges to
$\mu$ (resp. $|\mu|$) in the narrow topology of measures as well.
}\end{remark}

Next,  recall that for a compact set $K\subset\RR^n$, the 
capacity ${\rm Cap}_{1,\, s}(K)$, $1<s<+\infty$, of $K$
associated to the Sobolev space $W^{1,\, s}(\RR^n)$ is defined
by
\begin{equation*}
{\rm Cap}_{1, \, s}(K)=\inf\Big\{\int_{\RR^n}(|\nabla \varphi|^s +\varphi^s) dx: \varphi\in C^\infty_0(\RR^n),
\varphi\geq \chi_K \Big\},
\end{equation*}
where $\chi_{K}$ is the characteristic function of $K$.

\begin{definition} Given  $s >1$ and a  domain $\Om\subset\RR^n$ we define the space $M^{1, \, s}(\Om)$ to be the set of all  signed measures
$\mu$ with bounded total variation in $\Om$ such that the quantity $[\mu]_{M^{1,\, s}(\Om)}<+\infty$, where
$$[\mu]_{M^{1,\, s}(\Om)}:=\sup\left\{ |\mu|(K\cap\Om)/{\rm Cap}_{1,\, s}(K): {\rm Cap}_{1, \, s}(K)>0 \right\},$$
with the supremum being taken over all compact sets $K\subset\RR^n$.
\end{definition}

\begin{theorem}\label{mainconsequence}
Let $2-1/n<p\leq n$ and  $m>1$.  Suppose that $\mathcal{A}$ satisfies \eqref{sublinear}-\eqref{ellipticity}. 
Then there exist constants $s=s(n, p, \alpha, \beta)>1$ and $\delta=\delta(n, p, \alpha, \beta, m)\in (0, 1)$ such that the following holds. 
If  $[\mathcal{A}]_{s}^{R_0}\leq \delta$ and $\Om$ is $(\delta, R_0)$-Reifenberg flat for some $R_0>0$, then for any renormalized solution $u$ to 
the boundary value problem  \eqref{basicpde} we have 
\begin{equation*}
\left[|\nabla u|^\frac{m(p-1)}{m-1}\right]_{M^{1, \, m}(\Om)} \leq C_0\,  [\mu]_{M^{1, \, m}(\Om)}^{\frac{m}{m-1}}.
\end{equation*}
Here  $C_0$ depends only on $n, p, \alpha, \beta, m, {\rm diam}(\Om)$,  and ${\rm diam}(\Om)/R_0$.
\end{theorem}

\begin{proof}
Let $s$ and $\delta$ be as in Theorem \ref{main} corresponding to $q=\frac{m(p-1)}{m-1}>0$.  Then from Lemma 3.1 in \cite{MV} and Theorem \ref{main} we find
$$\left[|\nabla u|^\frac{m(p-1)}{m-1}\right]_{M^{1, \, m}(\Om)} \leq C \left[{\rm\bf M}_1(\mu)^\frac{m}{m-1}\right]_{M^{1,\, m}(\Om)}.$$

Note that, since $\mu$ is zero outside $\Om$, with $x\in\Om$ we have 
$${\rm\bf M}_1(\mu)(x)\leq  \sup_{0<r\leq {\rm diam}(\Om)} \frac{r\, |\mu|(B_r(x))}{|B_r(x)|}=: {\rm\bf M}^{{\rm diam}(\Om)}_1(\mu)(x).$$

On the other hand, a basic result in \cite{MV} (Theorems 1.1 and 1.2) says that if $\mu$ is a measure in $M^{1,\, m}(\Om)$, $m>1$, then $({\rm \bf G}_1 * |\mu|)^{\frac{m}{m-1}}\in M^{1,\, m}(\RR^n)$ with a bound
$$\left[({\rm\bf G}_1* |\mu|)^{\frac{m}{m-1}}\right]^{\frac{m-1}{m}}_{M^{1,\, m}(\RR^n)}\leq C [\mu]_{M^{1,\, m}(\Om)};$$
see also \cite[Corollary 2.5]{Ph1}.  Here ${\rm \bf G}_1(x)$ is the Bessel kernel of order $1$ defined via its Fourier transform by $\widehat{{\rm \bf G}_1}(\xi)=(1+|\xi|^2)^{-1/2}$. As it is well-known that  
$${\rm\bf M}^{{\rm diam}(\Om) }_1(\mu) \leq C(n, {\rm diam}(\Om))\, {\rm \bf G}_1 * |\mu|,$$ 
we immediately obtain
$$\left[{\rm\bf M}^{{\rm diam}(\Om) }_1(\mu)^{\frac{m}{m-1}}\right]^{\frac{m-1}{m}}_{M^{1,\, m}(\Om)}\leq C [\mu]_{M^{1,\, m}(\Om)},$$
where the constant $C$ may also depend on ${\rm diam}(\Om)$. From this we obtain the theorem.
\end{proof}

In what follows, by a sequence of standard mollifiers $\{\rho_k\}_{k\geq 1}$ we mean functions
$$\rho_k(x)= k^n\rho(k x), \qquad k=1, 2, \dots,$$ 
where $\rho\in C^{\infty}_{0}(\RR^n)$, $\rho\geq 0$, and $\rho$ is radially decreasing with $\norm{\rho}_{L^1}=1$. 
\begin{lemma}\label{measureapprox} Let $\{\rho_k\}_{k\geq 1}$ be a sequence of standard mollifiers. Suppose that $\mu\in M^{1,\, m}(\Om)$ with $m>1$.
Then so is $\rho_k*\mu$ with
\begin{equation}\label{Aconst}
[\rho_k*\mu]_{M^{1,\, m}(\Om)}\leq A\, [\mu]_{M^{1,\, m}(\Om)},
\end{equation}
where $A=A(n, m)$.
\end{lemma}
\begin{proof}
With ${\rm \bf G}_1(x)$ being the Bessel kernel of order $1$ as above, we have 
$$ {\rm \bf G}_1*(\rho_k*|\mu|)= \rho_k*{\rm \bf G}_1* |\mu| \leq {\rm\bf M}({\rm \bf G}_1*|\mu|),$$
where ${\rm\bf M}$ is the Hardy-Littlewood maximal function. By Lemma 3.1 in \cite{MV} this gives 
\begin{equation*}
\left[({\rm \bf G}_1*(\rho_k*|\mu|))^{\frac{m}{m-1}}\right]_{M^{1, \, m}(\RR^n)} 
\leq C \left[({\rm \bf G}_1*|\mu|)^{\frac{m}{m-1}}\right]_{M^{1, \, m}(\RR^n)}.
\end{equation*}

The lemma then follows from Theorem 1.1 and 1.2 in \cite{MV} (see also \cite[Theorem 2.3]{Ph1}).
\end{proof}

We are now ready for the proof of Theorem \ref{main-Ric}.

\begin{proof}[{\bf Proof of Theorem \ref{main-Ric}}]
Suppose that 
\begin{equation}\label{smallness-Ric}
[\om]_{M^{1, \frac{q}{q-p+1}}(\Om)}\leq d_0:=\frac{q-p+1}{q}\left(\frac{p-1}{C_0 q}\right)^{\frac{p-1}{q-p+1}},
\end{equation}
where $C_0$ is the constant  in Theorem \ref{mainconsequence} corresponding to $m=q/(q-p+1)$.

Let $g: [0, \infty)\rightarrow \RR$ be a function defined by
$$g(t)=C_0\left(t+[\om]_{M^{1, \frac{q}{q-p+1}}(\Om)}\right)^{\frac{q}{p-1}}-t.$$

We have $g(0)>0$ and $\lim_{t\rightarrow\infty}g(t)=\infty.$ Moreover,
$g'(t)=0$ if and only if
$$t=t_0=\left(\frac{p-1}{C_0q}\right)^{\frac{p-1}{q-p+1}}-[w]_{M^{1, \frac{q}{q-p+1}}(\Om)}.$$

It is then easy to see from \eqref{smallness-Ric} that $t_0>0$ and that $g$ has exactly one root $T$ in the interval $(0, t_0]$.

We now let
$$E=\left\{ v\in W_{0}^{1,\, 1}(\Om):v\in W_{0}^{1,\, q}(\Om) {\rm ~and~} [|\nabla v|^q]_{M^{1, \frac{q}{q-p+1}}(\Om)}
\leq T \right\}.$$

By Fatou Lemma   $E$ is  closed  under the strong topology of $W_{0}^{1,\,1}(\Om)$.
Moreover, since $q>n(p-1)/(n-1)>1$ we find that $E$ is convex.

Suppose for now that $\om\in \mathcal{M}_0(\Om)$. We then consider a map $S:E\rightarrow E$ defined for each $v\in E$ by $S(v)=u$
where $u\in W_{0}^{1,\,q}(\Om)$ is the unique solution to

\begin{eqnarray*}
\left\{\begin{array}{rcl}
-{\rm div}\, \mathcal{A}(x, \nabla u) &=&|\nabla v|^q + \om \quad {\rm in}~ \Om,\\
u&=&0\quad {\rm on}~ \partial\Om.
\end{array}
\right.
\end{eqnarray*}

The uniqueness is guaranteed here since $|\nabla v|^q + \om \in \mathcal{M}_0(\Om)$ (see \cite{BGO, DMOP}). Also, note that $S(v)\in E$ for every $v\in E$ since by Theorem \ref{mainconsequence} and the fact that $T$ is a root of $g$
we have
\begin{eqnarray*}
[|\nabla(S(v))|^q]_{M^{1, \frac{q}{q-p+1}}(\Om)}
&\leq& C_0 [|\nabla v|^q +\om]_{M^{1, \frac{q}{q-p+1}}(\Om)}^{\frac{q}{p-1}}\\
&\leq& C_0 \left( T + [\om]_{M^{1, \frac{q}{q-p+1}}(\Om)} \right)^{\frac{q}{p-1}}\\
&=& T.
\end{eqnarray*}

By Lemma \ref{cont-and-comp} below  the map $S:E\rightarrow E$ is continuous and $S(E)$ is precompact
under the strong topology of $W_{0}^{1,\, 1}(\Om)$.  Thus it has a fixed point on $E$ by  Schauder Fixed Point Theorem (see, e.g., \cite[Corollary 11.2]{GTru}). This gives a solution $u$ to the problem \eqref{Riccati} as desired.

We now remove the assumption $\om\in \mathcal{M}_0(\Om)$. To this end, we write 
$$\om=\om^+ -\om^-=(\om^+)_0 +(\om^+)_s -(\om^-)_0 -(\om^-)_s=\om_0+\om^{+}_s-\om^{-}_s,$$
where $\om_0\in \mathcal{M}_0(\Om)$ and $\om_s\in \mathcal{M}_s(\Om)$. 
Since $\om\in M^{1, \frac{q}{q-p+1}}(\Om)$ so is each measure in the above decomposition. Moreover,
$$[\om_0]_{M^{1,\, \frac{q}{q-p+1}}(\Om)}+[\om^{+}_s]_{M^{1,\, \frac{q}{q-p+1}}(\Om)}+[\om^{-}_s]_{M^{1,\, \frac{q}{q-p+1}}(\Om)}\leq [\om]_{M^{1,\, \frac{q}{q-p+1}}(\Om)}.$$

Let 
$$\lambda^+_k= \rho_k *\om^+_s,\qquad \lambda^-_k= \rho_k *\om^-_s,$$
where $\{\rho_k\}_{k\geq 1}$ is a sequence of standard mollifiers. Then by Lemma \ref{measureapprox} the measures $\om_k:=\om_0+ \lambda^+_k-\lambda^-_k\in \mathcal{M}_0(\Om)$ 
and  satisfy
\begin{eqnarray*}
[\om_k]_{M^{1,\, \frac{q}{q-p+1}}(\Om)}&\leq& [\om_0]_{M^{1,\, \frac{q}{q-p+1}}(\Om)}\\
&& +\, A\, [\om^+_s]_{M^{1,\, \frac{q}{q-p+1}}(\Om)}+A\, [\om^-_s]_{M^{1,\, \frac{q}{q-p+1}}(\Om)}\\ 
&\leq& (1+A)[\om]_{M^{1,\, \frac{q}{q-p+1}}(\Om)},
\end{eqnarray*} 
where $A$  is the constant in \eqref{Aconst} (with $m=q/(q-p+1)$).
This yields
$$[\om_k]_{M^{1,\, \frac{q}{q-p+1}}}\leq d_0,$$
provided  
$$[\om]_{M^{1,\, \frac{q}{q-p+1}}}\leq c_0:= d_0/(1+A).$$

Thus from the above argument for each $k\geq 1$ there exists a renormalized solution $u_k\in W^{1, \, q}_{0}(\Om)$ to
\begin{eqnarray*}
\left\{\begin{array}{rcl}
-{\rm div}\, \mathcal{A}(x, \nabla u_k) &=&|\nabla u_k|^q + \om_k \quad {\rm in}~ \Om,\\
u_k&=&0\quad {\rm on}~ \partial\Om,
\end{array}
\right.
\end{eqnarray*}
with 
\begin{equation}\label{ukbound}
\left[|\nabla u_k|^q\right]_{M^{1, \, \frac{q}{q-p+1}}} \leq T. 
\end{equation}

In particular, there holds 
\begin{equation}\label{rhsbound}
\int_{\Om}|\nabla u_k|^q dx +|\om_k|(\Om)\leq C.
\end{equation}

With \eqref{rhsbound} in hand, it follows from the proof of Theorem 3.4 in \cite{DMOP} that there exists a subsequence $\{u_{k'}\}$
converging a.e. to  a function $u\in W^{1,\, q}_0(\Om)$  such  that 
\begin{equation}\label{ptws1}
\nabla u_{k'} \rightarrow \nabla u \quad {\rm a.e.~ in~} \Om.
\end{equation}

Using \eqref{ukbound}, \eqref{ptws1} and arguing as in the proof of Lemma \ref{cont-and-comp} below we find that 
$$|\nabla u_{k'}|^q \rightarrow |\nabla u|^q \quad {\rm strongly~in~} L^1(\Om).$$

Thus by the stability result of renormalized solutions \cite[Theorem 3.4]{DMOP}, the function $u$ is a solution of 
\eqref{Riccati} with the desired properties. 
\end{proof}

We are now left with the following lemma whose proof follows an idea from \cite{Ph1}.  Here one also needs to use the fact that $q>1$. 
See the proof of Theorem 3.4 in \cite{Ph1} for details.

\begin{lemma}\label{cont-and-comp} The map $S:E\rightarrow E$ is continuous and its image $S(E)$ is precompact under the strong topology of $W^{1, \, 1}_{0}(\Om)$.
\end{lemma}
\noindent{\bf Acknowledgments.} The author wishes to thank Professor Igor E. Verbitsky his helpful comments made on the Riccati type equation.

\end{document}